\documentclass[12pt]{article}
\usepackage{amsmath,amssymb,amsfonts,amsbsy,amsthm}
\usepackage{enumitem}
\usepackage[usenames]{color}
\usepackage[colorlinks=true,
linkcolor=webgreen,
filecolor=webbrown,
citecolor=webgreen]{hyperref}

\definecolor{webgreen}{rgb}{0,.5,0}
\definecolor{webbrown}{rgb}{.6,0,0}

\usepackage{float}

\theoremstyle{plain}
\newtheorem{theorem}{Theorem}
\newtheorem{corollary}[theorem]{Corollary}
\newtheorem{lemma}[theorem]{Lemma}

\theoremstyle{definition}
\newtheorem{definition}[theorem]{Definition}
\newtheorem{example}[theorem]{Example}

\theoremstyle{remark}
\newtheorem{remark}[theorem]{Remark}

\def\zz{\mathbb Z}
\def\nn{\mathbb N}

\newcommand{\beq}{\begin{equation}}
\newcommand{\eeq}{\end{equation}}

\newcommand{\seqnum}[1]{\href{https://oeis.org/#1}{\rm \underline{#1}}}

\begin{document}

\begin{center}
\vskip 1cm{\Large\bf 
Fixed Points and Cycles of the Kaprekar \\
\vskip .01in
Transformation: 2. Even bases}
\vskip 1cm
Anthony Kay\\
 \href{mailto: anthony.kay292@btinternet.com}{\tt  anthony.kay292@btinternet.com}\\
\ \\
 Katrina Downes-Ward
\end{center}

\vskip .2in

\begin{abstract}

We develop a classification of the fixed points and cycles of the Kaprekar transformation in even bases. The most numerous fixed points and cycles are those we denote \emph{symmetric} and \emph{almost-symmetric}; the structure of the cycles of these classes in base $b$ is determined by subgroups and cosets in the multiplicative group modulo $b-1$. We provide methods and formulae for enumerating the fixed points and cycles of these and other classes. A detailed survey of the fixed points and cycles is provided for bases $4$, $6$ and $8$, including a rigorous proof that the classification is complete in base $4$.

\end{abstract}

\section{Introduction}

A comprehensive introduction to the \emph{Kaprekar transformation} $T_{b,n}$ is given in our earlier paper \cite{kdw1}; here we just provide the briefest details of the transformation and our notation.

Given a positive integer $x$ which has $n$ digits when written in base $b$, we rearrange the digits in descending and ascending orders, and subtract the latter from the former. So, given $a_0,a_1,\ldots,a_{n-1}\in\zz$ with
$$0\le a_0\le a_1\le\cdots\le a_{n-1}<b,$$
where the $a_j(j=0,1,\ldots,n-1)$ are the digits of $x$ in ascending order, we have
\beq\label{defkt}
T_{b,n}(x):=\sum_{j=0}^{n-1}a_jb^j-\sum_{j=0}^{n-1}a_jb^{n-1-j};
\eeq
we refer to the integer $T_{b,n}(x)$ as the \emph{successor} of $x$. Iterating this procedure, we can define the $m$'th successor recursively by $T_{b,n}^{m+1}(x)=T_{b,n}(T_{b,n}^m(x))$. We must eventually reach a fixed point or a cycle, since there are finitely many $n$-digit integers in base $b$, for any choice of $b$ and $n$. The \emph{length} of a cycle arising from some given $x$ is the minimal positive integer $l$ such that $T_{b,n}^{m+l}(x)=T_{b,n}^m(x)$ for all $m\ge m_0$, for some $m_0\in\nn$; a fixed point may be regarded as a cycle of length $l=1$, but in this paper we use the word ``cycle'' strictly to denote cycles of length $l\ge2$. Our concern in this paper is with the fixed points and cycles that exist in even bases $b\ge4$; base $2$ was comprehensively analysed by Yamagami \cite{yama17}, and odd bases, which display very different behaviours from even bases, were the subject of our earlier paper \cite{kdw1}.

The result of applying the Kaprekar transformation to an integer $x$ depends only on its \emph{Kaprekar index}, defined as the $b$-tuple $\mathbf{k}=(k_0,k_1,\ldots,k_{b-1})$ where the \emph{component} $k_i$ is the number of occurrences of the digit $i$ in the set $\{a_0,a_1,\ldots,a_{n-1}\}$. So we shall often display only Kaprekar indices, rather than the integers that they represent. We write $x\mapsto T_{b,n}(x)$ and $\mathbf{k}\mapsto\mathbf{k'}$ to indicate the succession of integers and of Kaprekar indices, respectively; so if $\mathbf{k}$ is the Kaprekar index of $x$, then $\mathbf{k'}$, with components $(k_0',k_1',\ldots,k_{b-1}')$, is the Kaprekar index of $T_{b,n}(x)$. We also use multiple primes to denote further successors, for example, $\mathbf{k}''$ is the Kaprekar index of $T_{b,n}^2(x)$. A formula for a component of $\mathbf{k'}$ in terms of the components of $\mathbf{k}$, under given conditions on the latter, will be called a \emph{succession formula}. A notation convention adopted throughout this paper is that $i$ (and occasionally $i'$) refers to the digits $0,1,\ldots,b-1$ which may appear in an integer expressed in base $b$, while $j$ refers to the position of a digit in a string representing an integer in some base, as used in equation (\ref{defkt}) above.

If an integer $x$ is a \emph{repdigit}, with every digit the same in the base under consideration, then $T_{b,n}(x)=0$. We therefore exclude repdigits from consideration; no non-repdigit can be transformed to a repdigit in any even base. If every non-repdigit with some digit-count $n$ in some base $b$ reaches the same fixed point or cycle following sufficient iterations of the Kaprekar transformation, that fixed point or cycle is called \emph{unanimous}.

The remainder of this paper is set out as follows.  In Section \ref{sec2} we describe classes of fixed point and cycle that exist in even bases $b\ge4$, and general principles for enumerating the fixed points and cycles in each class. The next three sections contain lists of fixed points and cycles that are known to exist in bases $4$, $6$, and $8$, respectively, together with their enumerations. For base $4$ we prove that our list of fixed points and cycles is complete; the proof is by exhaustion, and it is not feasible to provide similar proofs for higher bases, although it is reasonable to conjecture that our lists for bases $6$ and $8$ are complete.  We draw some conclusions in Section \ref{sec6}.

\section{Fixed points and cycles in general even bases $b\ge4$}\label{sec2}

\subsection{General properties of the transformation}

To keep the notation tidy, we introduce the symbol
$$B:=b-1$$
for the greatest digit that appears in base $b$. We also use a variant of the difference notation introduced by Prichett, Ludington and Lapenta \cite{pll}:
\beq\label{diffdef}
d_j:=a_{n-1-j}-a_j\quad\mbox{for}\;\;j=0,1,\ldots,\nu,
\eeq
where $\nu=\lfloor n/2\rfloor-1$, and we observe that $d_0\ge d_1\ge\cdots\ge d_\nu$. We set $\mu=\max \{j:d_j>0\}$, so either $d_{\mu+1}=0$ or $\mu=\nu$.
Performing the subtraction (\ref{defkt}), the digits of $T_{b,n}(x)$ are then found to be \cite{pll}
\beq\label{ktres}
d_0,d_1,\ldots, d_{\mu-1},d_\mu-1,B,\ldots, B,B-d_\mu,B-d_{\mu-1},\ldots ,B-d_1,b-d_0,
\eeq
where commas have been inserted between digits for clarity since some digits are in the form of expressions involving a subtraction. We may immediately note one useful property of the Kaprekar transformation:
\begin{lemma}\label{divb}
(a) The sum of digits of $T_{b,n}(x)$ is divisible by $B$. 

(b) In terms of Kaprekar index components, $\sum_{i=1}^{B-1}ik_i'$ is divisible by $B$.
\end{lemma}
\begin{proof}
(a) This result may be checked directly from (\ref{ktres}). 

(b) From the definition of the Kaprekar index, the sum of digits of $T_{b,n}(x)$ is $\sum_{i=0}^B ik_i'$, divisible by $B$ from the result (a). The terms $0k_0'$ and $Bk_B'$ are obviously divisible by $B$, leaving the remainder of the sum divisible by $B$.
\end{proof}
Several important results on the Kaprekar index of a successor may be deduced from the digit string (\ref{ktres}): 
\begin{theorem}\label{noex}
(a) If an integer $x$ satisfies the condition $d_\mu=b-d_0$, then the Kaprekar index $\mathbf{k}'$ of its successor $T_{b,n}(x)$ has $k_B'\ge k_0'$ and $k_i'=k_{B-i}'$ for $i=1,2,\ldots,B-1$.

(b) Even when $d_\mu\ne b-d_0$, the components of $\mathbf{k}'$ satisfy 
$$|k_i'-k_{B-i}'|\le3\text{ for }i=1,2,\ldots,B-1.$$ 
The only possibility for $|k_i'-k_{B-i}'|=3$ is with $i=b/2$ or $i=b/2-1$, and requires either $d_0=b/2$ or $d_\mu=b/2$.

(c) The only case in which $k_B'< k_0'$ is when $\mu=\nu=n/2-1$, with $d_\mu=1$ and $d_0<B$; in this case, $k_B'=0$ and $k_0'=1$.

(d) $k_0'\le k_0$ except possibly if $k_0=0$, in which case $k_0'\le1$.
\end{theorem}
\begin{proof}
(a) The digit-string (\ref{ktres}) of the successor $T_{b,n}(x)$ consists of: a left section of $\mu+1$ digits $d_j$ $(j=0,1,\ldots,\mu)$, except that the last of these digits is $d_\mu-1$ rather than $d_\mu$; a right section of $\mu+1$ digits $B-d_j$ $(j=0,1,\ldots,\mu)$, except that the last of these digits is $b-d_0$ rather than $B-d_0$; and a central section of $n-2\mu-2$ digits $B$, which is absent or reduced to a single such digit if $\mu=\nu$ and $n$ is even or odd, respectively. If $d_\mu=b-d_0$, then $d_\mu-1=B-d_0$, and so \emph{every} occurrence of a digit $i$ in the left section has a corresponding occurrence of the digit $B-i$ in the right section: the two exceptional digits noted above compensate each other. The only remaining digits are the digits $B$ in the central section; so the only exception to $k_i'=k_{B-i}'$ is for $i=0$ and $i=B$, with $k_B>k_0$ unless the central section is absent.

(b) In (\ref{ktres}) the only cases where a digit $i$ in the left section does not have a corresponding digit $B-i$ in the right section are the pairs $(d_0,b-d_0)$ and $(d_\mu-1,B-d_\mu)$. It is possible for three (but not all four) of the digits in these pairs to be equal: either $d_0=b-d_0=B-d_\mu$ in which case $d_0=b/2$ and $k_{b/2}-k_{B-b/2}=3$ (note that $B-b/2=b/2-1$); or $b-d_0=d_\mu-1=B-d_\mu$ in which case $d_\mu=b/2$ and $k_{b/2}-k_{B-b/2}=-3$. In any other case, we may have up to two of the four digits equal to some common value $i^*$, for example $d_0=B-d_\mu=i^*$; so the count of such a digit would be up to two greater than that of its corresponding digit, i.e., we could have $k_{i^*}-k_{B-i^*}=2$.

(c) The condition $d_\mu=1$ ensures that there is one digit $0$ in the left section of (\ref{ktres}). The condition $d_0<B$ then ensures that $d_\mu\ne b-d_0$, so that the proof of part (a) does not apply; it also implies that there are no digits $B$ in the left section, and no digits $0$ in the right section. The condition $\mu=\nu=n/2-1$ (so $n$ is even) ensures that the central section of digits $B$ is absent. The conditions are all necessary: if any one of them fails, $k_B'\ge k_0'$.

(d) Suppose $a_j=0$ for $j=0,1,\ldots,k_0-1$, with $k_0\ge2$. It is then possible that $d_j=B$, so that $B-d_j=0$, for $j=1,\ldots,k_0-1$; and it is also possible that $d_\mu-1=0$. No other digits in (\ref{ktres}) can be $0$, so $k_0'\le k_0$. If $k_0=1$ or $k_0=0$, no digit in the right section of (\ref{ktres}) can be $0$; we may still have $d_\mu-1=0$, so $k_0'\le 1$. So it is only possible for $k_0'>k_0$ if $k_0=0$, with $k_0'=1$.
\end{proof}
The result that $k_i'=k_{B-i}'$ for $i=1,2,\ldots,B-1$ in part (a) of the theorem was proved in our earlier paper \cite{kdw1} under conditions which are certainly sufficient to ensure that $d_\mu=b-d_0$. We shall see that in most cases of interest, the condition $d_\mu=b-d_0$ is  satisfied with $a_0=0,a_{n-1}=B$ and $d_\mu=1$. When $d_\mu=b-d_0$, the Kaprekar index of (\ref{ktres}) is the same as if the left section consisted entirely of digits $d_j$ and the right section entirely of digits $B-d_j$, without exceptions; so in the theory below, we shall frequently omit any mention of the exceptional digits, for the sake of brevity in cases where we have shown the condition $d_\mu=b-d_0$ to apply.

Before considering specific types of fixed point and cycle, we note a general property of cycles that follows from Theorem \ref{noex}(d).
\begin{corollary}\label{k0const}
If $k_0>1$ in any member of a cycle, then $k_0$ remains constant throughout the cycle.
\end{corollary}
\begin{proof}
From Theorem \ref{noex}(d), the value of $k_0$ can only increase from $0$ to $1$. Therefore, if it decreases from a value greater than $1$, it can never return to that greater value in a later succession.
\end{proof}

\subsection{Full-index fixed points and cycles}

\begin{definition}\label{defsym}
(a) A Kaprekar index is \emph{full} if 
$$k_i\ge1\text{ for }i=0,1,\ldots,B,$$
i.e., if no component of the index is zero. 

A cycle is \emph{full} if every member has a full Kaprekar index.

(b) A full Kaprekar index is \emph{symmetric} if 
$$k_{B-i}=k_i\text{ for }i=0,1,\ldots,B.$$
A cycle is \emph{symmetric} if every member has a symmetric Kaprekar index.

(c) A full Kaprekar index is \emph{almost-symmetric} if 
$$k_{B-i}=k_i\text{ for }i=1,\ldots,B-1\text{ and  }k_0<k_B<k_0+\min\{k_i:1\le i\le B-1\}.$$
The \emph{degree of asymmetry} $\alpha$ is defined in this case as
$$\alpha:=k_B-k_0.$$
A cycle is \emph{almost-symmetric} if every member has an almost-symmetric Kaprekar index.
\end{definition}
Fixed points with full Kaprekar indices, either symmetric or almost-symmetric, account for the majority of fixed points of the Kaprekar transformation in even bases; and similarly the majority of cycles consist entirely of integers with full Kaprekar indices, either symmetric or almost-symmetric. We now develop the theory of these fixed points and cycles.

\begin{theorem}\label{symsucc}
In any even base $b\ge4$, the successor $\mathbf{k}'$ of a symmetric or almost-symmetric Kaprekar index $\mathbf{k}$ has the following properties:-

(a) $\mathbf{k}'$ is symmetric or almost-symmetric according to whether $\mathbf{k}$ is symmetric or almost-symmetric.

(b) $k_0'=k_0$ and $k_B'=k_B$.

(c) $k_{2i}'=k_{B-2i}'=k_i$ for $1\le i\le (B-1)/2$.
\end{theorem}
\begin{proof}
If $\mathbf{k}$ is full, then $a_0=0$ and $a_{n-1}=B$, so $d_0=B$ and $b-d_0=1$.

Consider first the symmetric case. In an even base, $\mathbf{k}$ being symmetric implies: (i) $n$ is even; (ii) $\mu=\nu=n/2-1$, so the central section of digits $B$ in (\ref{ktres}) is absent; (iii) $a_{n/2-1}=b/2-1$ and $a_{n/2}=b/2$, so that $d_\mu=1$. The condition $d_\mu=b-d_0$ is now satisfied, so Theorem \ref{noex}(a) applies: $k_i'=k_{B-i}'$ for $i=1,2,\ldots,B-1$, and because there is no central section of digits $B$, $k_B'=k_0'$ also. Hence $\mathbf{k}'$ is symmetric.

Since $1\le d_j\le B$ for $0\le j\le\mu$, the only digits in (\ref{ktres}) that can be equal to $B$ are in the left section (the central digits $B$ being absent). A difference $d_j$ can only take the value $B$ if $a_j=0$ and $a_{n-1-j}=B$; and with $\mathbf{k}$ being symmetric, the numbers of these digits $0$ and $B$ are both equal to $k_B$. Hence there are $k_B$ instances of differences $d_j=B$ in (\ref{ktres}), so $k_B'=k_B$. The symmetry of $\mathbf{k}$ and $\mathbf{k}'$ then yields $k_0'=k_0$.

With $\mathbf{k}$ being symmetric, $a_{n-1-j}=B-a_j$ for $0\le j\le \nu$, so from (\ref{diffdef}), $d_j=B-2a_j$. So for each instance of a digit $a_j=i\le(B-1)/2$ in $x$, there is an odd digit $d_j=B-2i$ in the left section of (\ref{ktres}) and an even digit $B-d_j=2i$ in the right section of (\ref{ktres}). This accounts for all digits of $T_{b,n}(x)$ when $\mathbf{k}$ is symmetric, so $k_{2i}'=k_{B-2i}'=k_i$.

Now consider the almost symmetric case, in which $k_B=k_0+\alpha$ where the degree of asymmetry satisfies $1\le\alpha<\min\{k_i:1\le i\le B-1\}$: see Definition \ref{defsym}. For $0\le j\le k_0-1$, we have $a_j=0$ and $a_{n-1-j}=B$, so there are $k_0$ instances where $d_j=B$. For $k_0\le j\le k_B-1$, we have $a_j=1$ and $a_{n-1-j}=B$, so there are $k_B-k_0=\alpha$ instances where $d_j=B-1$. For $k_B\le j\le k_0+k_1-1$, we have $a_j=1$ and $a_{n-1-j}=B-1$, so there are $k_0+k_1-k_B=k_1-\alpha$ instances where $d_j=B-2$. Continuing likewise, there are $\alpha$ instances where $d_j=B-i$ for each odd $i$, and $k_i-\alpha$ instances where $d_j=B-2i$ for each $i\le(B-1)/2$. Observe that $B-i$ for all odd $i$ yields all possible even values of $d_j$, while $B-2i$ for all $i\le(B-1)/2$ yields all possible odd values of $d_j$.

Moving on to the right section of (\ref{ktres}), we now have $\alpha$ instances where $B-d_j=i$ for each odd $i$, and $k_i-\alpha$ instances where $B-d_j=2i$ for each $i\le(B-1)/2$, the latter covering all possible even digits. Combining digit counts from the left and right sections, we now have $k_i$ instances of each odd digit $B-2i$ for $i\le(B-1)/2$ (from $k_i-\alpha$ instances in the left section and $\alpha$ instances in the right section), and $k_i$ instances of each even digit $2i$ for $i\le(B-1)/2$ (from $\alpha$ instances in the left section and $k_i-\alpha$ instances in the right section). This yields the result in part (c) of the theorem, and, for part (a), that $k_{B-i}'=k_i'$ for $1\le i\le B-i$. The digit $0$ can only arise in the right section of (\ref{ktres}), as $B-d_j$ with $d_j=B$, of which there are $k_0$ instances; hence $k_0'=k_0$. It remains to observe that $\mu=\sum_{i=0}^{(B-1)/2}k_i-1$, from which the almost-symmetry yields that the central section of (\ref{ktres}) consists of $\alpha$ digits $B$. Hence $k_B'=k_0'+\alpha=k_B$, and the verification of almost-symmetry is complete.
\end{proof}

A large class of fixed points emerges from the above analysis:
\begin{theorem}\label{ffp}
In any even base $b\ge4$, a Kaprekar index with $k_1=k_2=\cdots=k_{B-1}\ge1$ and $1\le k_0\le k_B<k_0+k_1$ represents a fixed point of the Kaprekar transformation.
\end{theorem}
\begin{proof}
The conditions of this theorem imply that $\mathbf{k}$ is symmetric or almost symmetric, so Theorem \ref{symsucc} applies. For $1\le i\le (B-1)/2$, the values of $2i$ and $B-2i$ cover all digits $i'$ such that $1\le i'\le B-1$; so from Theorem \ref{symsucc}(c), if 
all $k_i$ are equal for $1\le i\le (B-1)/2$, then all $k_{i'}'$ are equal for $1\le i'\le B-1$, and their common value is equal to the common value of the $k_i$. Part (b) of Theorem \ref{symsucc} shows that $k_0$ and $k_B$ are fixed, completing the proof.
\end{proof}

Theorem \ref{ffp} yields the classes of \emph{symmetric fixed points} (when $k_B=k_0$) and \emph{almost-symmetric fixed points} (when $k_0< k_B<k_0+k_1$). In base ten, both of these classes can be identified as subclasses of Prichett et al.'s \cite{pll} Class A and of the third class listed in Dolan's \cite{dolan} Theorem 1. But our Theorem \ref{symsucc} yields not only the fixed points of Theorem \ref{ffp}, but also a proliferation of symmetric and almost-symmetric cycles of a particular length, as given in Theorem \ref{symcyc} below. That theorem uses the following definition which was introduced by Yamagami and Matsui \cite{yama18}, together with some related theory which was developed in our previous paper \cite{kdw1}:
\begin{definition}\label{defsigma}
For any odd integer $r\ge3$, let $\sigma(r)$ be the least positive integer such that $2^{\sigma(r)}\equiv\pm1(\!\!\!\!\mod r)$. See \seqnum{A003558} in OEIS.
\end{definition}
\begin{remark}\label{sigdiv}
If $r^*$ is a divisor of $r$, then $\sigma(r^*)$ is a divisor of $\sigma(r)$.
\end{remark}
Given $i_0\in\nn$ and odd $B\in\nn$ with $i_0<B/2$ and $\gcd(i_0,B)=1$, define $i_m$ recursively by 
\beq\label{qrules}
i_{m+1}=\begin{cases}
2i_m,&\text{ if }2i_m<\frac{B}{2};\\
B-2i_m,&\text{ if }2i_m>\frac{B}{2}.
\end{cases}
\eeq
Clearly $i_m\equiv\pm2^mi_0$ (mod $B$), so that by Definition \ref{defsigma}, $i_m=i_0$ when $m=\sigma(B)$ and not for any smaller value of $m$. The integers in the $i$\emph{-cycles} generated by this procedure are precisely $q_i/2$, where $q_i$ are integers in the $q$-cycles defined in our earlier paper \cite{kdw1}; as explained there, the cycles are derived from subgroups generated by $2$, or their cosets, in $\zz_B^*$, the multiplicative group of integers modulo $B$. In case $\gcd(i_0,B)=c>1$, the cycle containing $i_0$ will be derived from a subgroup or coset in $\zz_{B/c}^*$, and will have length $\sigma(B/c)$ rather than $\sigma(B)$; according to Remark \ref{sigdiv}, $\sigma(B/c)$ is a divisor of $\sigma(B)$. Thus, for any odd $B\ge3$, the integers $\{1,\ldots,(B-1)/2\}$ can be partitioned into $i$-cycles, which are ordered subsets with cardinalities equal to $\sigma(B)$ or a divisor of $\sigma(B)$.
\begin{theorem}\label{symcyc}
In an even base $b$, suppose a Kaprekar index $\mathbf{k}$ is either symmetric or almost symmetric, with not all the $k_i(i=1,2,\ldots,B-1)$ equal to each other. Then $\mathbf{k}$ is a member of a cycle with length equal to $\sigma(B)$ or a divisor of $\sigma(B)$; this includes the possibility of a fixed point (a cycle of length $1$).
\end{theorem}
\begin{proof}
First observe that according to Theorem \ref{symsucc}(b), the components $k_0$ and $k_B$ in symmetric and almost-symmetric Kaprekar indices remain invariant under the Kaprekar transformation. Thus, given [almost-]symmetry, we only need be concerned with components $k_1,\ldots,k_{(B-1)/2}$.

Theorem \ref{symsucc}(c) gives the succession formula $k_{B-2i}'=k_{2i}'=k_i$ for $1\le i\le (B-1)/2$. Observing that either $i'=2i$ or $i'=B-2i$ must satisfy $1\le i'\le (B-1)/2$, and comparing with (\ref{qrules}), we see that in an $m$'th successor, $k_{i_m}^{(m)}=k_{i_0}$ where $i_m$ is in the $i$-cycle starting at $i_0$; in particular, $i_m=i_0$ when $m=\sigma(B)$. Thus for every $i_0$ with $1\le i_0\le(B-1)/2$, $k_{i_0}^{(\sigma(B))}=k_{i_0}$ (with possibly $k_{i_0}^m=k_{i_0}$ for $m$ a smaller divisor of $\sigma(B)$ if $\gcd(i_0,B)>1$). Thus the $\sigma(B)$'th successor of a Kaprekar index $\mathbf{k}$ is equal to that original Kaprekar index. Hence $\mathbf{k}$ will in general be in a cycle of length $\sigma(B)$; however, there are exceptional circumstances in which the cycle length is a smaller divisor of $\sigma(B)$, as follows.

We may partition the components $k_1,\ldots,k_{(B-1)/2}$ into ordered subsets according to the partition of the integers $1,\ldots,(B-1)/2$ into $i$-cycles described above. The subsets are of size equal to $\sigma(B)$ or any divisor of $\sigma(B)$. First suppose that all subsets of size $\sigma(B)$ contain components which form $d$ identical sub-cycles of length $\sigma(B)/d$ for some divisor $d$ (possibly $d=\sigma(B)$, with all components in the subset being identical); then the Kaprekar indices may be in a cycle of length $\sigma(B)/d$. However, we also need to consider any subsets which are themselves of size $\sigma(B)/c$ for some divisor $c$, and which would generate cycle length $\sigma(B)/c$. The possibilities are illustrated in Example \ref{symcb64}.
\end{proof}
\begin{example}\label{symcb10}
In base $10$, with $B=9$, there is an $i$-cycle $(1,2,4)$ of length $\sigma(9)=3$, while the integer $i=3$, for which $\gcd(3,B)=3$, consitutes an $i$-cycle of length $\sigma(9/3)=1$. Consider an [almost-]symmetric Kaprekar index $\mathbf{k}$ with components 
$$k_1=k_8=t,k_2=k_7=u,k_3=k_6=v,k_4=k_5=w$$ 
(and $k_0\le k_9<k_0+\min\{t,u,v,w\})$: given the $i$-cycles $(1,2,4)$ and $(3)$, the successor will have components 
$$k_1'=k_8'=w,k_2'=k_7'=t,k_3'=k_6'=v,k_4'=k_5'=u,$$ 
the next successor will have components 
$$k_1''=k_8''=u,k_2''=k_7''=w,k_3''=k_6''=v,k_4''=k_5''=t,$$
and in the third successor the components will return to their values in $\mathbf{k}$.

For example, suppose that $\mathbf{k}=(1,4,3,4,4,4,4,3,4,2)$; note that $k_i=k_{9-i}$ for $1\le i\le8$ and that $k_0=1$, $k_9=2$, and $\min\{k_i:1\le i\le 8\}=3$, so the requirements for almost-symmetry in Definition \ref{defsym}(c) are fulfilled. Note also that although $k_1=k_4$, it is not true that $k_1=k_2=k_4$. The succession of Kaprekar indices as described above will be
\begin{align*}
(1,4,3,4,4,4,4,3,4,2)\mapsto(1,4,4,4,3,3,4,4,4,2)\mapsto&(1,3,4,4,4,4,4,4,3,2)\mapsto\\
&\qquad(1,4,3,4,4,4,4,3,4,2):
\end{align*}
a cycle of length $l=3=\sigma(9)$.

Now consider the Kaprekar index $\mathbf{k}=(1,4,4,3,4,4,3,4,4,2)$. Here, $k_1=k_2=k_4$, so going through the $i$-cycle will not change the values of these components; and the remaining component $k_3$ is in an $i$-cycle of length $1$, so never changes. Hence this Kaprekar index is a fixed point.
\end{example}
\begin{example}\label{symcb64}
For a case with a more intricate system of $i$-cycles, consider base $64$. We require the $i$-cycles for $B=63$. There are three of length $\sigma(63)=6$, consisting of integers coprime with $63$:
$$(1,2,4,8,16,31);\quad(5,10,20,23,17,29);\quad(11,22,19,25,13,26).$$
Integers $i$ with $\gcd(i,63)=3$ appear in an $i$-cycle of length $\sigma(63/3)=6$:
$$(3,6,12,24,15,30).$$
Next, $\sigma(63/7)=3$ and $\sigma(63/9)=3$, so there are two $i$-cycles of length $3$, consisting respectively of integers $i$ with $\gcd(i,63)=7$ and $\gcd(i,63)=9$:
$$(7,14,28);\quad(9,18,27).$$
Finally, the integer $21$ is in an $i$-cycle of length $1$. We have now accounted for the $31$ integers $i$ with $1\le i\le(63-1)/2$. 

Now consider a symmetric or almost-symmetric Kaprekar index in base $64$. Its components $k_i$ ($1\le i\le31$) are partitioned into ordered subsets of cardinalities $6$, $3$ or $1$, corresponding to the $i$-cycles, e.g., $(k_1,k_2,k_4,k_8,k_{16},k_{31})$ corresponding to the first $i$-cycle listed above. For nearly every choice of values of its components, the Kaprekar index will be in a cycle of length $6$. But suppose the components in the subsets corresponding to each of the four $i$-cycles of length $6$ have equalities between $1$'st and $4$'th, between $2$'nd and $5$'th, and between $3$'rd and $6$'th components: so in the first subset,
$$k_1=k_8,k_2=k_{16},k_4=k_{31},$$
and similarly in the other subsets,
$$k_5=k_{23},k_{10}=k_{17},k_{20}=k_{29};\quad k_{11}=k_{25},k_{22}=k_{13},k_{19}=k_{26};\quad k_3=k_{24},k_6=k_{15},k_{12}=k_{30}.$$
Then these ordered subsets of components would each consist of two identical cycles of length $3$, and since the only other subsets have cardinalities $3$ or $1$, the Kaprekar index would be in a cycle of length $3$.

Alternatively, suppose that the subsets of $6$ components each consist of three identical cycles of length $2$. So in the first subset,
$$k_1=k_4=k_{16},k_2=k_8=k_{31},$$
with similar equalities in the other three such subsets. This suggests that the Kaprekar index may be in a cycle of length $2$; but there are also subsets of $3$ components, so in fact the Kaprekar index would be in a cycle of length $6$, unless the components in each subset of cardinality $3$ were equal, i.e., unless
\beq\label{subs3}
k_7=k_{14}=k_{28}\text{ and } k_9=k_{18}=k_{27}.
\eeq
Finally, there would be the possibility of the Kaprekar index being a fixed point, if every subset consisted of equal components, i.e., if we had equalities, 
$$k_1=k_2=k_4=k_8=k_{16}=k_{31};\quad k_5=k_{10}=k_{20}=k_{23}=k_{17}=k_{29};$$
$$k_{11}=k_{22}=k_{19}=k_{25}=k_{13}=k_{26};\quad  k_3=k_6=k_{12}=k_{24}=k_{15}=k_{30}$$
in the subsets of cardinality $6$, as well as the equalities (\ref{subs3}).
\end{example}

It is interesting to see how the subgroup generated by $2$ in a multiplicative group of integers modulo some $r$ determines both the [almost-]symmetric cycles in even bases and the regular cycles in bases $b\equiv3$ (mod $4$) that were discussed in our earlier paper \cite{kdw1}, but in completely different ways. In the odd bases, $r$ was equal to $n-2$ or a divisor thereof (for primitive cycles), so cycle length was determined by digit-count $n$; but in even bases, $r=b-1$, so that in any given base, nearly all symmetric and almost-symmetric cycles have the same length, $\sigma(b-1)$, as observed by Myers \cite{myers2}. However, whereas in the odd bases, fixed points behaved strictly as cycles of length $l=1$, so could only arise with $r=3$, in even bases fixed points do not require $\sigma(B)=1$, and they exist in every even base.  The conjecture by Myers \cite{myers2} that in any \emph{given} even base there are only finitely many cycle lengths remains plausible; but the range of cycle lengths that exists among \emph{all} even bases is unbounded, assuming Artin's conjecture on primitive roots to be correct.  

\subsection{Zero-free fixed points}

There are two classes of fixed points that appear in all even bases $b\ge4$ and which have $k_0=0$, so no zeroes appear in the base-$b$ expression of the integer.

For each $t\in\nn$, there is a \emph{uniform zero-free fixed point}, with $k_0=0$ and $k_1=k_2=\cdots=k_B=t$; so the digit-count is $n=(b-1)t$. Referring to equation (\ref{diffdef}), we have $a_{n-1-j}=b-a_j$ so that $d_j=b-2a_j$ for all $j=0,1,\ldots,\mu$, where $\mu=t(B-1)/2-1$ and each digit from $1$ to $B$ appears $t$ times among the $a_j$; so there are $t$ instances of each even value of difference from $2$ to $b-2$ among the $d_j(0\le j\le\mu)$. Thus these even digits each appear $t$ times in the left section of (\ref{ktres}), while the odd digits from $1$ to $b-3$ each appear $t$ times in the right section of (\ref{ktres}); except for $d_\mu-1=1$ on the left replacing a digit $2$ which appears instead as $b-d_0$ on the right. In the central section of (\ref{ktres}) there are $t$ instances of the digit $B$. To summarise, (\ref{ktres}) contains $t$ instances of each digit from $1$ to $B$, as required for the fixed point. In base ten, these uniform zero-free fixed points fall within Prichett et al.'s \cite{pll} Class A and the second class listed in Dolan's \cite{dolan} Theorem 1.

For each $t\in\nn$, there is a \emph{triad fixed point}, with $k_{b/2-1}=k_{b/2}=k_{b-1}=t$ and all other $k_i=0$; so the digit-count is $n=3t$. The differences are $d_j=(b-1)-(b/2-1)=b/2$ for $j=0,1,\ldots,t-1$, yielding $t-1$ instances of the digit $b/2$ together with the digit $d_\mu-1=b/2-1$ in the left section of (\ref{ktres}), $t-1$ instances of the digit $B-b/2=b/2-1$ together with the digit $b-d_0=b/2$ in the right section of (\ref{ktres}), and $t$ instances of the digit $B$ in the central section. In base ten, these triad fixed points fall within Prichett et al.'s \cite{pll} Class D, and are the first class listed in Dolan's \cite{dolan} Theorem 1.

There do exist other classes of fixed points with $k_0=0$ in even bases. A universal feature is that $k_{b/2-1}=k_{b/2}=k_{b-1}$, but values of the remaining components of the Kaprekar indices are specified in various ways, and there do not appear to be consistent patterns over all even bases $b\ge6$. We shall therefore describe these classes of zero-free fixed points  in the sections of this paper dealing with the individual bases.

\subsection{Other fixed points and cycles}

The remaining fixed points and cycles that we have found may be classified as follows.

\emph{Non-symmetric $\sigma$-cycles.} These are controlled by $i$-cycles in a manner similar to the symmetric and almost-symmetric cycles, so have the same length $l=\sigma(B)$ (hence the designation ``$\sigma$-cycles''); but instead of the succession rule $k_{2i}'=k_{B-2i}'=k_i$ which applies in the case of [almost-]symmetry, there are rules of the form 
\beq\label{kdel}
k_{i'}'=k_i+\delta_i,
\eeq
where $\delta_i$ is an integer specific to each digit $i$, and
$$i'=\begin{cases}
2i,&\text{ if }0<i<\frac{B}{4};\\
B-2i,&\text{ if }\frac{B}{4}<i<\frac{B}{2};\\
2B-2i,&\text{ if }\frac{B}{2}<i<\frac{3B}{4};\\
2i-B,&\text{ if }\frac{3B}{4}<i<B.\\
\end{cases}$$
This specification of $i$\emph{-successors} is an extension of the rules (\ref{qrules}) to allow for non-symmetry, with $k_i'\ne k_{B-i}'$ in general; it yields the same $i$-cycles as (\ref{qrules}) when $i<B/2$, and also corresponding $i$-cycles when $i>B/2$. Clearly, the sum of the $\delta_i$ over each $i$-cycle must be zero to allow a Kaprekar cycle of length $\sigma(B)$, and this is indeed found in the non-symmetric $\sigma$-cycles displayed below in bases $6$ and $8$. Other features common to all cycles of this type that we have found (but which we have not proved to be necessary) are: (i) $k_0=1$ and $k_B=0$ in every member of a cycle; (ii) $-2\le\delta_i\le2$ for each digit $i$; (iii) $\sum_{i=0}^{(B-1)/2}k_i=\sum_{i=(B+1)/2}^B k_i$ in every cycle member.

\emph{Single-parameter fixed points and cycles.} The cycles may be of any length, with each fixed point or member of a cycle being given by a formula involving a single parameter $t$, which may take any positive integer value. The following features appear in every single-parameter fixed point or cycle that we have found, but have not been proved to be necessary: (i)  in fixed points, $k_0=1$ and $k_B=0$, while in cycles, every member has either $k_0=0$ or $k_0=1$, with both options appearing at least once in the cycle; (ii) other components are either fixed integers ($0$, $1$ or occasionally $2$) or of the form $t+\delta$, where $-1\le\delta\le3$. Although no cycle member can be symmetric or almost-symmetric, the bound on asymmetry from Theorem \ref{noex}(b), \mbox{$|k_i-k_{B-i}|\le3$} for $i=1,2,\ldots,B-1$, means that $k_{B-i}$ depends on $t$ if and only if $k_i$ depends on $t$ (except for the case $i=0$).

\emph{Special cycles.} For small digit-counts $n$, there also exist in each base a few individual fixed points and/or cycles which do not fit into any class described above; we refer to these fixed points and cycles as \emph{special}.

\subsection{Enumeration of fixed points and cycles}

For some of the classes of fixed points and cycles which we have described above, it is possible to derive general formulae for the number of fixed points or cycles of that class that exist with any given digit-count $n$ in all even bases $b\ge4$. These formulae are applied to the individual bases $4$, $6$ and $8$ in the relevant sections below which deal with those bases, where we also provide enumerations of fixed points and cycles of those classes for which no general formulae can be derived.

\subsubsection{Symmetric and almost-symmetric fixed points}

Here we enumerate fixed points of the type specified in Theorem \ref{ffp}, with \mbox{$k_1=k_2=\cdots=k_{B-1}$}. Fixed points with symmetric or almost-symmetric Kaprekar indices can arise with not all the $k_i(i=1,2,\ldots,B-1)$ equal to each other, as described in Theorem \ref{symcyc} and Examples \ref{symcb10} and \ref{symcb64}, and enumerating these would require the methods discussed below for symmetric and almost-symmetric cycles.

In base $b$, symmetric fixed points have digit-count $n=2k_0+(b-2)k_1$ with $k_0\ge1$ and $k_1\ge1$, so that $k_1\le(n-2)/(b-2)$. Thus there are
\beq\label{ensfp}
N_S=\left\lfloor\frac{n-2}{b-2}\right\rfloor
\eeq
symmetric fixed points with even digit-count $n$, and none with odd $n$ since $b$ is even.

For almost-symmetric fixed points, we have defined $\alpha=k_B-k_0$, and we also define $\beta=k_B+k_0$ so that the digit count is 
\beq\label{nask1}
n=(b-2)k_1+\beta.
\eeq
Note that $\alpha$, $\beta$ and $n$ all have the same parity. From the requirement for almost-symmetry that $1\le k_0<k_B<k_0+k_1$, we obtain that $\alpha\ge1$, $\beta\ge \alpha+2$, and $k_1\ge \alpha+1$, so that (\ref{nask1}) yields 
\beq\label{ndineqgen}
n\ge(b-1)\alpha+b.
\eeq 
Hence almost-symmetric fixed points may exist for odd $n\ge2b-1$ and even $n\ge3b-2$, with values of $\beta$ ranging from $\alpha+2$ to $n-(b-2)(\alpha+1)$ in steps of $b-2$; this gives 
$$\left\lfloor\frac{n-(b-1)\alpha-2}{b-2}\right\rfloor$$ 
values of $\beta$. For odd $n$, let $\alpha=2m-1$, so that 
$$m\le\frac{n-1}{2(b-1)}$$ 
from (\ref{ndineqgen}), and there are 
$$\left\lfloor\frac{n-2(b-1)m+b-3}{b-2}\right\rfloor$$ 
allowed values of $\beta$ for each such $m$; thus the total number of almost-symmetric fixed points for each odd $n$ is
\beq\label{ascountogen}
N_A=\sum_{m=1}^{\lfloor(n-1)/2(b-1)\rfloor}\left\lfloor\frac{n-2(b-1)m+b-3}{b-2}\right\rfloor.
\eeq
Similarly for even $n$, we can let $\alpha=2m$ and the number of almost-symmetric fixed points is
\beq\label{ascountegen}
N_A=\sum_{m=1}^{\lfloor(n-b)/2(b-1)\rfloor}\left\lfloor\frac{n-2(b-1)m-2}{b-2}\right\rfloor.
\eeq
Closed-form expressions for these sums depend on congruence of $n$ modulo $(b-1)(b-2)$, so evaluations are done below for individual bases. However, from (\ref{ascountogen}) and (\ref{ascountegen}), we see that $N_A$ takes the same value for an odd digit-count $n$ as for the greater even digit-count $n+b-1$: we can write
\beq\label{asevenodd}
N_A(n+b-1)=N_A(n)\quad\text{for odd $n$.}
\eeq
Furthermore, for even $n$, the combined number of symmetric and almost-symmetric fixed points is
$$N_{SA}:=N_S+N_A=\sum_{m=0}^{\lfloor(n-b)/2(b-1)\rfloor}\left\lfloor\frac{n-2(b-1)m-2}{b-2}\right\rfloor,$$
from which (by making the substitution $m=m^*-1$)
\beq\label{sasevenodd}
N_{SA}(n-(b-1))=N_A(n)\quad\text{for odd $n$.}
\eeq
Thus, if we evaluate $N_{SA}$ for each even $n$, we can then obtain $N_A$ for each odd $n$ using (\ref{sasevenodd}) (note that for odd $n$, $N_S=0$ so that $N_{SA}=N_A$), and then obtain $N_A$ for even $n$ from (\ref{asevenodd}). Since the expressions (\ref{ensfp}), (\ref{ascountogen}) and (\ref{ascountegen}) are increasing functions of digit-count $n$, there are more almost-symmetric fixed points with odd digit-count than for an adjacent even digit-count, but the total number of symmetric and almost-symmetric fixed points is greater for even than for adjacent odd digit-counts.

\subsubsection{Other classes of fixed points}

While there is a single triad fixed point for each digit-count $n$ that is divisible by $3$, and a single uniform zero-free fixed point for each $n$ that is divisible by $b-1$, other zero-free fixed points are more numerous in even bases $b\ge6$; these are enumerated below for individual bases. Single-parameter fixed points occur singly for certain values of digit-count (specified as $n\equiv d$ (mod $(b-2)$) for some $d<b-2$ in those bases that we have examined). Special fixed points occur only for particular digit-counts; these are again covered below for individual bases.

\subsubsection{Symmetric and almost-symmetric cycles}\label{enumsasc}

Consider first a base $b$ for which $\sigma(B)=(B-1)/2$, where $B=b-1$, so there is only a single $i$-cycle. This is the simplest case for enumeration, and in particular includes bases $4$, $6$ and $8$ which are discussed in more detail below; a necessary but not sufficient condition for it to apply is that $B$ is prime. In such a base, the only symmetric or almost-symmetric fixed points will be those with $k_1=k_2=\cdots=k_{B-1}$, discussed above; and if $\sigma(B)$ is prime, then all symmetric and almost-symmetric cycles will have length $l=\sigma(B)$ since the ordered subset of Kaprekar index components corresponding to the $i$-cycle cannot be split into several identical sub-cycles.

For ease of notation, let
$$C:=\frac{B-1}{2}=\frac{b}{2}-1.$$
With a single $i$-cycle, there are $C$ parameters $t_1,\ldots,t_C$, the values of which cycle through the Kaprekar index components $k_1,\ldots,k_C$. [Note that in Example \ref{symcb10} and in the discussion of bases $6$ and $8$ below, we use the notation $t,u,v$ rather than $t_1,t_2,t_3$.] With $\alpha=k_B-k_0$, we define $\gamma_c=t_c-\alpha$ for $c=1,\ldots,C$. The conditions of Theorem \ref{symcyc} (including the conditions for [almost-]symmetry) can be formulated as: $k_0\ge1$; either $\alpha=0$ (for symmetry) or $\alpha\ge1$ (for almost-symmetry); $\gamma_c\ge1$ for $c=1,\ldots,C$; and not all the $\gamma_c$ are equal to each other.

The digit-count is
\beq\label{digctsas}
n=k_0+k_B+2\sum_{c=1}^C t_c=2\left(k_0+\sum_{c=1}^C \gamma_c\right)+B\alpha.
\eeq
To satisfy the requirement that not all the $\gamma_c$ are equal, the minimum digit-count occurs with $\gamma_1=2$ and $\gamma_c=1$ for $c\ne1$; thus $n\ge b+2$ for a symmetric cycle, and 
\beq\label{mindigct}
n\ge b+2+(b-1)\alpha
\eeq
for an almost-symmetric cycle with given $\alpha\ge1$.

The enumeration then proceeds as follows. We first require a formula for $N_\gamma(C;\omega)$, the number of distinct cycles of $\gamma_c$ values (not all equal) with $\sum_{c=1}^C \gamma_c=\omega$ for any given $\omega$; details of this calculation are given below. We then need $N_{k\gamma}(C;\eta)$, the number of cycles with $k_0+\sum_{c=1}^C \gamma_c=\eta$, for any given $\eta$: since $k_0\ge1$ and $\omega=\sum_{c=1}^C \gamma_c\ge1$, this is
\beq\label{neta}
N_{k\gamma}(C;\eta)=\sum_{\omega=1}^{\eta-1}N_\gamma(C;\omega).
\eeq
From (\ref{digctsas}) we have \mbox{$\eta=(n-B\alpha)/2$}. Thus, by setting $\alpha=0$, the number of symmetric cycles for even $n$ is simply 
$$N_{SC}=N_{k\gamma}\left(C;\frac{n}{2}\right)$$
(and $N_{SC}=0$ for odd $n$). For almost-symmetric cycles, we note from (\ref{mindigct}) that $\alpha\le(n-b-2)/(b-1)$ and has the same parity as $n$. Thus for odd $n$, set $\alpha=2m-1$ so that $m\le(n-3)/2(b-1)$ and $\eta=(n-2Bm+B)/2$, and the number of almost-symmetric cycles is
\beq\label{nacodd}
N_{AC}=\sum_{m=1}^{\lfloor(n-3)/2(b-1)\rfloor}N_{k\gamma}\left(C;\frac{n-2Bm+B}{2}\right).
\eeq
For even $n$, we set $\alpha=2m$ and obtain
\beq\label{naceven}
N_{AC}=\sum_{m=1}^{\lfloor(n-b-2)/2(b-1)\rfloor}N_{k\gamma}\left(C;\frac{n-2Bm}{2}\right),
\eeq
and the total number of symmetric and almost-symmetric cycles is
\beq\label{nsac}
N_{SAC}=N_S+N_A=\sum_{m=0}^{\lfloor(n-b-2)/2(b-1)\rfloor}N_{k\gamma}\left(C,\frac{n-2Bm}{2}\right).
\eeq
From (\ref{nacodd}) and (\ref{naceven}) we see that the relations (\ref{asevenodd}) and (\ref{sasevenodd}) between the numbers of symmetric and almost-symmetric fixed points for even and odd digit-counts also apply to almost-symmetric cycles: we can write,
\beq\label{ascevenodd}
N_{AC}(n+b-1)=N_{AC}(n)=N_{SAC}(n-(b-1))\quad\text{for odd $n$.}
\eeq

Formulae for the above sums in closed form for general $b$ are not available, mainly because of the intricacies in the first stage of calculation. To find $N_\gamma(C;\omega)$, we first require $\widetilde{N_\gamma}(C;\omega)$, the number of ordered $C$-tuples of positive integers $(\gamma_1,\ldots,\gamma_C)$ that add to a given $\omega$. For $C=2$, this is $\widetilde{N_\gamma}(2;\omega)=\omega-1$, and since a $(C+1)$'th integer can be combined with any $C$-tuple with a sum no greater than $\omega-1$, we have the recursive formula, 
$$\widetilde{N_\gamma}(C+1;\omega)=\sum_{\theta=1}^{\omega-1}\widetilde{N_\gamma}(C;\theta).$$ 
By induction, it can be verified that this yields $\widetilde{N_\gamma}(C;\omega)$ as the binomial coefficient,
\beq\label{ngamt}
\widetilde{N_\gamma}(C;\omega)=\frac{(\omega-1)(\omega-2)\cdots(\omega-(C-1))}{(C-1)!}={\omega-1 \choose C-1}.
\eeq
Now, the members of a cycle of $\gamma_c$ values are the cyclic permutations of the $C$-tuple $(\gamma_1,\ldots,\gamma_C)$, so the number of cycles that we have found is $\widetilde{N_\gamma}(C;\omega)/C$; except that if $\omega$ is divisible by $C$ (the only case in which $\widetilde{N_\gamma}(C;\omega)$ may \emph{not} be divisible by $C$), we must exclude the $C$-tuple with $\gamma_1=\cdots=\gamma_C$ which corresponds to a fixed point according to Theorem \ref{ffp}. So the number of cycles of $\gamma_c$ values is 
\beq\label{ngamma}
N_\gamma(C;\omega)=\left\lfloor\frac{\widetilde{N_\gamma}(C;\omega)}{C}\right\rfloor.
\eeq
This is correct if $C$ is prime; but if $C$ is composite, some $C$-tuples $(\gamma_1,\ldots,\gamma_C)$ are in the form of $d$ identical sub-cycles of length $C/d$. So in such cases, a division by $C/d$ rather than $C$ is required to obtain the number of cycles from the number of $C$-tuples of that type. [For example, with $\omega=10$ and $C=4$, we have $\widetilde{N_\gamma}(4;10)=9\times8\times7/3!=84$; but there are four quadruples,
$$(1,4,1,4);\quad(2,3,2,3);\quad(3,2,3,2);\quad(4,1,4,1)$$
which are members of cycles of length $2$. So we have $N_\gamma(4;10)=80/4+4/2=22$.] This complication precludes the presentation of any completely general formula for $N_\gamma(C,\omega)$.

Finally, we recall that the above calculations apply to bases for which there is a single $i$-cycle. Where there are multiple $i$-cycles, we have 
\beq\label{etamom}
\eta=k_0+\omega_1+\cdots+\omega_G
\eeq 
for some $G>1$, where $\omega_g$ ($g=1,\ldots,G$) is the sum of the $\gamma_c$ in the $g$'th $i$-cycle. So (\ref{neta}) must be replaced by a calculation which allows for all possible combinations of $k_0$ and the $\omega_g$ that add to the given $\eta$. We do not explore the possibilities further here.

\subsubsection{Other classes of cycles}

Non-symmetric $\sigma$-cycles relate to $i$-cycles in a similar way to symmetric and almost-symmetric cycles, and (in all the examples that we have found) have an element of symmetry in that $\sum_{i=0}^{(B-1)/2}k_i=\sum_{i=(B+1)/2}^B k_i$, but have $k_0$ and $k_B$ fixed. Thus, in a base where there is only a single $i$-cycle, the number of non-symmetric $\sigma$-cycles will be simply the number of cycles of $C$ integers that add to a fixed $\omega=\sum_{i=1}^{C}k_i$; but the conditions on those integers may not be simply that they are positive. So calculations similar to that of $N_\gamma(C;\omega)$ for \mbox{[almost-]symmetric} cycles will be required, but with details depending on the specification of the particular non-symmetric $\sigma$-cycle. We provide details of the enumerations in the sections dealing with individual bases.

Similarly to single-parameter fixed points, single-parameter cycles occur singly for certain values of digit-count given by some linear formula in the parameter, and special cycles only occur for particular values of digit-count. These classes of cycle are again enumerated below for individual bases.

\section{Fixed points and cycles in base 4}\label{sec3}

We first list the fixed points and cycles that exist in base $4$, according to the classification in Section \ref{sec2}. We then enumerate them, and finally we provide a rigorous proof that our list is complete.

\subsection{List of fixed points and cycles}\label{b4fpc}

For base $4$, sequence \seqnum{A165021} in OEIS gives fixed points and least members of cycles (from which the remaining members of cycles may be computed), with sequence \seqnum{A165022} giving the lengths of cycles. From these data, we have classified the fixed points and cycles as follows.
\begin{description}
\item[Symmetric fixed points.] For every $k_0,k_1\in\nn$, the integer 
$$\underbrace{3\cdots3}_{k_0}\;\underbrace{1\cdots1}_{k_1-1}0\;\underbrace{2\cdots2}_{k_1}\;\underbrace{0\cdots0}_{k_0-1}\;1,$$
with Kaprekar index 
$$(k_0,k_1,k_1,k_0),$$
is a fixed point of the Kaprekar transformation, with digit-count $n=2(k_0+k_1)$, so symmetric fixed points exist with even digit-counts $n\ge4$.
\item[Almost-symmetric fixed points.] Given $k_0,k_1,k_3$ with $1\le k_0<k_3<k_0+k_1$ and $\alpha=k_3-k_0$, the integer
$$\underbrace{3\cdots3}_{k_0}\;\underbrace{2\cdots2}_{\alpha}\;\underbrace{1\cdots1}_{k_1-\alpha-1}\;0\;\underbrace{3\cdots3}_{\alpha}\;\underbrace{2\cdots2}_{k_1-\alpha}\;\underbrace{1\cdots1}_{\alpha}\;\underbrace{0\cdots0}_{k_0-1}\;1,$$
with Kaprekar index
$$(k_0,k_1,k_1,k_3),$$
is an almost-symmetric fixed point with digit-count $n=k_0+2k_1+k_3$. Since the conditions require $k_0\ge1$, $k_3\ge2$, and $k_0+k_1\ge k_3+1\ge k_0+2$ so that $k_1\ge2$, the minimum digit-count for an almost-symmetric fixed point is $n=7$.
\item[Triad fixed points.] There is no distinction between uniform zero-free and triad fixed points in base $4$. For each $t\in\nn$, the integer
$$\underbrace{2\cdots2}_{t-1}\;1\;\underbrace{3\cdots3}_{t}\;\underbrace{1\cdots1}_{t-1}\;2,$$
with Kaprekar index $(0,t,t,t)$, is a fixed point with digit-count $n=3t$.
\end{description}
Since $B=3$ and $\sigma(3)=1$, there are no symmetric or almost-symmetric cycles of length $l>1$ in base $4$.
\begin{description}
\item[Single-parameter cycles.] For each $t\in\nn$, there is a cycle of length $l=3$ and digit-count $n=3t+8$. The cycle of integers is
\begin{align*}
3\underbrace{2\cdots2}_{t-1}110\underbrace{3\cdots3}_{t+2}222\underbrace{1\cdots1}_{t} \mapsto 3\underbrace{2\cdots2}_{t+1}1\underbrace{3\cdots3}_{t+2}\underbrace{1\cdots1}_{t+3}&\mapsto\underbrace{2\cdots2}_{t+3}0\underbrace{3\cdots3}_{t}2\underbrace{1\cdots1}_{t+2}2\\
& \mapsto3\underbrace{2\cdots2}_{t-1}110\underbrace{3\cdots3}_{t+2}222\underbrace{1\cdots1}_{t} ,
\end{align*}
with Kaprekar indices
\beq\label{regcyc}
(1,t+2,t+2,t+3)\mapsto(0,t+4,t+1,t+3)\mapsto(1,t+2,t+5,t)\mapsto(1,t+2,t+2,t+3).
\eeq
\item[Special cycles.] There are four special cycles  in base $4$. We display the cycles of both the integers and their Kaprekar indices:
\begin{itemize}
\item With $n=2$, a cycle of length $l=2$:
\beq\label{scyc2}
03\mapsto21\mapsto03;\qquad(1,0,0,1)\mapsto(0,1,1,0)\mapsto(1,0,0,1).
\eeq
\item With $n=4$, a cycle of length $l=2$:
\beq\label{scyc4}
2022\mapsto1332\mapsto2022;\qquad(1,0,3,0)\mapsto(0,1,1,2)\mapsto(1,0,3,0).
\eeq
\item With $n=5$, a cycle of length $l=2$:
\beq\label{scyc5}
20322\mapsto23331\mapsto20322;\qquad(1,0,3,1)\mapsto(0,1,1,3)\mapsto(1,0,3,1).
\eeq
\item With $n=8$, a cycle of length $l=3$:
\begin{align}22033212\mapsto31333311&\mapsto22133112\mapsto22033212;\nonumber\\&(1,1,4,2)\mapsto(0,3,0,5)\mapsto(0,3,3,2)\mapsto(1,1,4,2).\label{scyc8}
\end{align}
\end{itemize}
\end{description}
We may observe that no member of the single-parameter cycle or any of the special cycles is symmetric or almost symmetric: in any Kaprekar index with $k_1=k_2\ge1$ in these cycles, either $k_0=0$ or $k_3\ge k_0+k_1$. On the other hand, there are several examples in these cycles of Kaprekar indices with $|k_2-k_1|=3$, which is allowed according to Theorem \ref{noex}(b): from the proof of that theorem, $k_2-k_1=3$ only if the predecessor in the cycle has $d_0=b/2=2$, and $k_2-k_1=-3$ only if the predecessor has $d_\mu=2$, as may be verified in the cycles displayed above.

From the above list of fixed points and cycles, we note that the special cycles for $n=2$ and $n=5$, the triad fixed point for $n=3$ and the almost-symmetric fixed point for $n=7$ (with Kaprekar index $(1,2,2,2)$) are unanimous. For digit-counts $n\le5$ in any base, Trigg \cite{trigg2,trigg3,trigg4,trigg5} has shown how an exhaustive list of fixed points and cycles may be obtained from examination of a manageable number of representative cases; for base $4$, the unanimous cycle with $n=2$, the triad fixed point with $n=3$, the symmetric fixed point and the cycle with $n=4$, and the unanimous cycle with $n=5$ may be found in Trigg's papers. Since his method is exhaustive, no further verification is needed of the absence of any other fixed points or cycles for $n\le5$, although this will be checked from the analysis below for general $n$.

\subsection{Enumeration of fixed points and cycles}

In base $b=4$, the formula (\ref{ensfp}) for the number of symmetric fixed points becomes
$$N_S=\frac{n-2}{2},$$
noting that symmetric fixed points exist only for even $n$. The formulae (\ref{ascountogen}) and (\ref{ascountegen}) for the numbers of almost-symmetric fixed points with odd and even digit-counts become
\beq\label{ascounto}
N_A=\sum_{m=1}^p\frac{n-6m+1}{2}=\frac{p}{2}(n-3p-2)\quad\text{ where }p=\lfloor(n-1)/6\rfloor\quad\text{($n$ odd)}
\eeq
and
\beq\label{ascounte}
N_A=\sum_{m=1}^p\frac{n-6m-2}{2}=\frac{p}{2}(n-3p-5)\quad\text{ where }p=\lfloor(n-4)/6\rfloor\quad\text{($n$ even)}.
\eeq
The only other fixed points in base $4$ are the triad fixed points that exist when $n$ is divisible by $3$. Evaluating (\ref{ascounto}) and (\ref{ascounte}) according to the congruence of $n$ modulo $6$, we can write down a set of formulae for the total number $N(n)$ of fixed points for any digit-count $n$; the annotations S, A, T indicate whether there are any symmetric, almost-symmetric or triad fixed points, respectively:
\begin{align*}
\text{When }n\equiv0\text{ (mod $6$): } &N(n)=\frac{n-2}{2}+\frac{(n-6)(n-4)}{24}+1\quad&\mbox{[S, A, T]}\\
\text{When }n\equiv1\text{ (mod $6$): } &N(n)=\frac{(n-1)(n-3)}{24}\quad&\mbox{[A]}\\
\text{When }n\equiv2\text{ (mod $6$): } &N(n)=\frac{n-2}{2}+\frac{(n-8)(n-2)}{24}\quad&\mbox{[S, A]}\\
\text{When }n\equiv3\text{ (mod $6$): }&N(n)=\frac{(n-3)(n-1)}{24}+1\quad&\mbox{[A, T]}\\
\text{When }n\equiv4\text{ (mod $6$): } &N(n)=\frac{n-2}{2}+\frac{(n-4)(n-6)}{24}\quad&\mbox{[S, A]}\\
\text{When }n\equiv5\text{ (mod $6$): } &N(n)=\frac{(n-5)(n+1)}{24}\quad&\mbox{[A]}
\end{align*}
For comparison with formulae presented below for higher bases, we may observe that the combined number of symmetric and almost-symmetric fixed points is given by the formula
$$N_{SA}:=N_S+N_A=\left\lfloor\frac{n(n+2)}{24}\right\rfloor\quad\text{for even $n$},$$
with $N_{SA}(n)=N_A(n+3)=N_A(n+6)$ for even $n$, and $N_{SA}(n)=N_A(n)$ for odd $n$: see equations (\ref{asevenodd}) and (\ref{sasevenodd}).

The only cycles in base $4$ are a special cycle with each of the digit-counts $n=2,4,5$, and $8$, and a single-parameter cycle with each digit-count $n\equiv2$ (mod $3$) for $n\ge11$.

\subsection{Proof that all fixed points and cycles have been found}

Our method requires an exhaustive list of succession formulae, giving components of $\mathbf{k}'$ in terms of components of $\mathbf{k}$, dependent on conditions (equalities and inequalities) satisfied by the latter. However, before displaying such a list, we note from Theorem \ref{noex}(c) that no Kaprekar index with $k_0>k_3$ can be a fixed point or a cycle member, unless $k_3=0$ and $k_0=1$; this reduces the number of cases to be considered. The following theorem arranges the remaining cases into classes labelled A--G, based mainly on which (if any) component(s) of $\mathbf{k}$ is/are zero. The cases within each class are labelled (i), (ii), $\ldots$; cases where $k_0>k_3>0$, or where $k_0>1$ and $k_3=0$, are omitted since $\mathbf{k}$ cannot then be a fixed point or cycle member.
\begin{theorem}\label{ksucc}Succession formulae in base $4$ are as follows.
\begin{enumerate}[label=\Alph*.]
\item Suppose $k_i\ge1$ for $i=0,1,2,3$, and $k_0=k_3$. Then:
\begin{enumerate}[label=(\roman*)]
\item if  $k_1>k_2$, then: $k_0'=k_3,\; k_1'=k_2'=k_2,\; k_3'=k_3+k_1-k_2;$
\item if  $k_1=k_2$, then: $k_0'=k_3,\; k_1'=k_2'=k_2,\; k_3'=k_3;$
\item if $k_1<k_2$, then: $k_0'=k_3,\; k_1'=k_2'=k_1,\; k_3'=k_3+k_2-k_1.$
\end{enumerate}
\item Suppose $k_i\ge1$ for $i=0,1,2,3$, and $k_0<k_3$. Then:
\begin{enumerate}[label=(\roman*)]
\item if $k_3<k_0+k_1<k_2+k_3$, then: $k_0'=k_0,\; k_1'=k_2'=k_1,\; k_3'=k_3+k_2-k_1;$
\item if $k_0+k_1=k_2+k_3$, then: $k_0'=k_0,\; k_1'=k_2'=k_1,\; k_3'=k_0;$
\item if $k_0+k_1>k_2+k_3$, then: $k_0'=k_0,\; k_1'=k_2'=k_3+k_2-k_0,\; k_3'=2k_0+k_1-k_2-k_3;$
\item if $k_0+k_1=k_3$, then: $k_0'=k_0-1,\; k_1'=k_1+2,\; k_2'=k_1-1,\; k_3'=k_0+k_2;$
\item if $k_0+k_1<k_3<k_0+k_1+k_2$, then: $k_0'=k_0,\; k_1'=k_2'=k_3-k_0,$

\qquad\qquad\qquad\qquad\qquad\qquad\qquad\qquad\qquad\qquad\qquad $k_3'=2k_0+k_1+k_2-k_3;$
\item if $k_3=k_0+k_1+k_2$, then: $k_0'=k_0,\; k_1'=k_2'=k_1+k_2,\; k_3'=k_0;$
\item if $k_3>k_0+k_1+k_2$, then: $k_0'=k_0,\; k_1'=k_2'=k_1+k_2,\; k_3'=k_3-k_1-k_2.$
\end{enumerate}
\item Suppose $k_i\ge1$ for $i=1,2,3$, and $k_0=0$. Then:
\begin{enumerate}[label=(\roman*)]
\item if $k_3<k_1<k_2+k_3$, then: $k_0'=1,\; k_1'=k_1-2,\; k_2'=k_1+1,\; k_3'=k_3+k_2-k_1;$
\item if $k_1=k_2+k_3$, then: $k_0'=1,\; k_1'=k_1-2,\; k_2'=k_1+1,\; k_3'=0;$
\item if $k_1>k_2+k_3$, then: $k_0'=1,\; k_1'=k_2+k_3-2,\; k_2'=k_2+k_3+1,\; k_3'=k_1-k_2-k_3;$
\item if $k_1=k_3$, then: $k_0'=0,\; k_1'=k_2'=k_3,\; k_3'=k_2;$
\item if $k_1<k_3<k_1+k_2$, then: $k_0'=1,\; k_1'=k_3-2,\; k_2'=k_3+1,\; k_3'=k_1+k_2-k_3;$
\item if $k_1+k_2=k_3$, then: $k_0'=1,\; k_1'=k_3-2,\; k_2'=k_3+1,\; k_3'=0;$
\item if $k_1+k_2<k_3$, then: $k_0'=1,\; k_1'=k_1+k_2-2,\; k_2'=k_1+k_2+1,\; k_3'=k_3-k_2-k_1.$
\end{enumerate}
\item Suppose $k_i\ge1$ for $i=0,2,3$, and $k_1=0$. Then:
\begin{enumerate}[label=(\roman*)]
\item if $k_0=k_3$, then: $k_0'=k_0-1,\; k_1'=k_2'=1,\; k_3'=k_0+k_2-1;$
\item if $k_0<k_3<k_0+k_2$, then: $k_0'=k_0,\; k_1'=k_2'=k_3-k_0,\; k_3'=2k_0+k_2-k_3.$
\item if $k_0+k_2=k_3$, then: $k_0'=k_0,\; k_1'=k_2'=k_2,\; k_3'=k_0;$
\item if $k_0+k_2<k_3$, then: $k_0'=k_0,\; k_1'=k_2'=k_2,\; k_3'=k_3-k_2.$
\end{enumerate}
\item Suppose $k_i\ge1$ for $i=0,1,3$, and $k_2=0$. Then:
\begin{enumerate}[label=(\roman*)]
\item if $k_0=k_3$, then: $k_0'=k_0-1,\; k_1'=k_2'=1\; k_3'=k_0+k_1-1;$
\item if $k_0<k_3<k_0+k_1$, then: $k_0'=k_0-1,\; k_1'=k_3-k_0+2,\; k_2'=k_3-k_0-1,$

\qquad\qquad\qquad\qquad\qquad\qquad\qquad\qquad\qquad\qquad $k_3'=2k_0+k_1-k_3;$
\item if $k_0+k_1=k_3$, then: $k_0'=k_0-1,\; k_1'=k_1+2,\; k_2'=k_1-1,\; k_3'=k_0;$
\item if $k_0+k_1<k_3$, then: $k_0'=k_0-1,\; k_1'=k_1+2,\; k_2'=k_1-1,\; k_3'=k_3-k_1.$
\end{enumerate}
\item Suppose $k_i\ge1$ for $i=0,1,2$, and $k_3=0$. Then:
\begin{enumerate}[label=(\roman*)]
\item if $k_0=k_2$, then: $k_0'=0,\; k_1'=k_2'=k_2,\; k_3'=k_1;$
\item if $k_0<k_2<k_0+k_1$, then: $k_0'=1,\; k_1'=k_2-2,\; k_2'=k_2+1,\; k_3'=k_0+k_1-k_2;$
\item if $k_0+k_1=k_2$, then: $k_0'=1,\; k_1'=k_2-2,\; k_2'=k_2+1,\; k_3'=0;$
\item if $k_0+k_1<k_2$, then: $k_0'=1,\; k_1'=k_0+k_1-2,\; k_2'=k_0+k_1+1,\; k_3'=k_2-k_1-k_0.$
\end{enumerate}
\item Cases with two components of $\mathbf{k}$ equal to zero:
\begin{enumerate}[label=(\roman*)]
\item If $k_0=k_1=0<k_2<k_3$, then: $k_0'=1,\; k_1'=k_2'=k_2-1,\; k_3'=k_3-k_2+1;$
\item If $k_0=k_1=0<k_2=k_3$, then: $k_0'=1,\; k_1'=k_2'=k_2-1,\; k_3'=1;$
\item If $k_0=k_1=0<k_3<k_2$, then: $k_0'=1,\; k_1'=k_2'=k_3-1,\; k_3'=k_2-k_3+1;$
\item If $k_0=k_2=0<k_1<k_3$, then: $k_0'=0,\; k_1'=k_2'=k_1,\; k_3'=k_3-k_1;$
\item If $k_0=k_2=0<k_1=k_3$, then: $k_0'=0,\; k_1'=k_2'=k_1,\; k_3'=0;$
\item If $k_0=k_2=0<k_3<k_1$, then: $k_0'=0,\; k_1'=k_2'=k_3,\; k_3'=k_1-k_3;$
\item If $k_0=k_3=0<k_1<k_2$, then: $k_0'=1,\; k_1'=k_2'=k_1-1,\; k_3'=k_2-k_1+1;$
\item If $k_0=k_3=0<k_1=k_2$, then: $k_0'=1,\; k_1'=k_2'=k_1-1,\; k_3'=1;$
\item If $k_0=k_3=0<k_2<k_1$, then: $k_0'=1,\; k_1'=k_2'=k_2-1,\; k_3'=k_1-k_2+1;$
\item If $k_1=k_2=0<k_0<k_3$, then: $k_0'=k_0-1,\; k_1'=k_2'=1,\; k_3'=k_3-1;$
\item If $k_1=k_2=0<k_0=k_3$, then: $k_0'=k_0-1,\; k_1'=k_2'=1,\; k_3'=k_0-1;$
\item If $k_1=k_3=0<k_0<k_2$, then: $k_0'=0,\; k_1'=k_2'=k_0,\; k_3'=k_2-k_0;$
\item If $k_1=k_3=0<k_0=k_2$, then: $k_0'=0,\; k_1'=k_2'=k_0,\; k_3'=0;$
\item If $k_2=k_3=0<k_0<k_1$, then: $k_0'=1,\; k_1'=k_2'=k_0-1,\; k_3'=k_1-k_0+1;$
\item If $k_2=k_3=0<k_0=k_1$, then: $k_0'=1,\; k_1'=k_2'=k_0-1,\; k_3'=1.$
\end{enumerate}
\end{enumerate}
\end{theorem}
\begin{proof}
Rather than using (\ref{ktres}), it is best to verify each case by setting out the subtraction as in the following example, which is case B(i) in which $k_0<k_3<k_0+k_1<k_2+k_3$. 
\begin{align}
&\underbrace{3\cdots\cdots\cdots3}_{k_3}\;\underbrace{2\cdots\cdots\cdots\cdots\cdots\cdots\cdots2}_{k_2}\;\underbrace{1\cdots\cdots\cdots\cdots\cdots1}_{k_1}\;\underbrace{0\cdots\cdots0}_{k_0}\nonumber\\
-&\underbrace{0\cdots0}_{k_0}\underbrace{1\cdots\cdots\cdots\cdots\cdots1}_{k_1}\;\underbrace{2\cdots\cdots\cdots\cdots\cdots\cdots2}_{k_2}\;\;\underbrace{3\cdots\cdots\cdots\cdots\cdots3}_{k_3}\nonumber\\
=\;&\underbrace{3\cdots3}_{k_0}\underbrace{2\cdots2}_{k_3-k_0}\;\underbrace{1\cdots\cdots1}_{k_0+k_1-k_3-1}0\;\underbrace{3\cdots\cdots\cdots3}_{k_3+k_2-k_1-k_0}\;\underbrace{2\cdots\cdots2}_{k_0+k_1-k_3}\;\,\underbrace{1\cdots\cdots1}_{k_3-k_0}\;\,\underbrace{0\cdots0}_{k_0-1}\;1.\label{subtex}
\end{align}
The total count of each digit in the result yields the components of $\mathbf{k}'$:
\begin{align*}
k_0'&=1+(k_0-1)=k_0,\\
k_1'&=(k_0+k_1-k_3-1)+(k_3-k_0)+1=k_1,\\
k_2'&=(k_3-k_0)+(k_0+k_1-k_3)=k_1,\\
k_3'&=k_0+(k_3+k_2-k_1-k_0)=k_3+k_2-k_1,
\end{align*} 
as given in case B(i) in the theorem. All other cases may be verified similarly.

Note that any Kaprekar index with three components equal to zero represents a repdigit, so is considered trivial.
\end{proof}

We first examine the possibility of fixed points among the cases listed in Theorem \ref{ksucc}.
\begin{theorem}\label{fpb4}
The only fixed points in base $4$ are the symmetric, almost-symmetric and triad fixed points described in Subsection \ref{b4fpc}.
\end{theorem}
\begin{proof}
Case A(ii) is a symmetric fixed point, since $k_3'=k_3=k_0=k_0'$ and $k_1'=k_2'=k_2=k_1$.

Case B(i) is an almost-symmetric fixed point if $k_1=k_2$, since we then have $k_0'=k_0$, $k_3'=k_3$, $k_1'=k_2'=k_1=k_2$ and $k_0<k_3<k_0+k_1$.

Case C(iv) is a triad fixed point if $k_1=k_2$, since we then have $k_0'=k_0=0$ and $k_1=k_2=k_3=k_1'=k_2'=k_3'$.

In every other case, setting $k_i'=k_i$ for $i=0,1,2,3$ leads to an inconsistency. This may be an inconsistency between the formulae for $k_i'$ and the conditions on $k_i$, for example, in case A(i) the result $k_1'=k_2'$ is inconsistent with the condition $k_1>k_2$ when $k_i'=k_i$; or setting $k_i'=k_i$ may simply make the succession formulae internally inconsistent, for example, in case B(iv) where $k_0'=k_0-1$. Other cases are left to the reader to verify.
\end{proof}

Eliminating the possibility of cycles other than the single-parameter and special cycles described in Subsection \ref{b4fpc} is not so straightforward, and will be done in several stages.
\begin{theorem}\label{nofull}
There do not exist any full cycles in base $4$.
\end{theorem}
\begin{proof}
By Definition \ref{defsym}(a), every member of a full cycle in base $4$ would have $k_i\ge1$ for $i=0,1,2,3$; so every member would be in class A or B in Theorem \ref{ksucc}.

Any symmetric Kaprekar index in base $4$ is a fixed point, so is not a cycle member. The successor $\mathbf{k}'$ is symmetric in cases A(ii), B(ii) and B(vi), so these cases are eliminated from consideration.

In case B(iv), $k_0'<k_0$ so according to Corollary \ref{k0const}, this case can only be within a cycle if $k_0=1$ and $k_0'=0$. But then the successor $\mathbf{k}'$ would not be full.

Next observe that $k_1'=k_2'$ in every case in classes A and B except B(iv); but we have just shown that case B(iv) cannot exist in a full cycle. So we require $k_1=k_2$ in every member of a full cycle, which eliminates cases A(i) and A(iii).

Given that $k_1=k_2$, case B(i) is an almost-symmetric fixed point (so not a cycle member). Next, the condition $k_0+k_1>k_2+k_3$ in case B(iii) becomes $k_0>k_3$, which cannot occur in a cycle member. Setting $k_1=k_2$ in case B(v), together with the conditions for that case, yields $k_0'<k_3'<k_0'+k_1'$, so that $\mathbf{k}'$ is almost-symmetric and hence is a fixed point.

The only remaining possibility of a full cycle is if all members are case B(vii). But $k_3'<k_3$ in this case, so no such cycle can exist.
\end{proof}
\begin{corollary}\label{fullbiv}
The only possibility for a cycle to include any member(s) with a full Kaprekar index is if one such member is case B(iv), with $k_0=1$.
\end{corollary}
\begin{proof}
From Theorem \ref{nofull}, a cycle containing a member with a full Kaprekar index must also contain a member whose index is not full. Every case in classes A and B  in Theorem \ref{ksucc} has a full successor $\mathbf{k}'$, except for case B(iv) with $k_0=1$. [Verification is left to the reader; note that the conditions for each case need to be invoked.]
\end{proof}

Before proceeding further, we make an observation that will be useful in restricting the number of cases to be considered in the analysis.
\begin{lemma}\label{kspm3}
In every succession in base $4$, either $k_2'=k_1'$ or $k_2'=k_1'-3$ or $k_2'=k_1'+3$.
\end{lemma}
\begin{proof}
Applying Lemma \ref{divb}(b) to base $4$, we have that $k_1'+2k_2'$ is divisible by $3$, and hence $k_1'-k_2'$ is divisible by $3$. Theorem \ref{noex}(b) then implies that the only possible values of $k_1'-k_2'$ are $0$, $3$ or $-3$.
\end{proof}
The lemma implies that in a cycle, \emph{every} member must have $k_1$ and $k_2$ either equal or differing by $3$.

\begin{theorem}\label{b4biv}
The only cycles in base $4$ which include a member with a full Kaprekar index are the special cycle (\ref{scyc8}) and the single-parameter cycles (\ref{regcyc}). 
\end{theorem}
\begin{proof}
According to Corollary \ref{fullbiv}, to prove this theorem we only need to consider the succession from a Kaprekar index in case B(iv) of Theorem \ref{ksucc}, with $k_0=1$. Lemma \ref{kspm3} restricts consideration to indices with $k_2=k_1$ or $k_2=k_1\pm3$.  

First consider the option that $k_2=k_1$, so $\mathbf{k}=(1,k_1,k_1,k_1+1)$ in case B(iv). If $k_1\ge3$, this index is in the single-parameter cycle (\ref{regcyc}), with $t=k_1-2$. If $k_1=2$, then Theorem \ref{ksucc}, cases B(iv), C(ii) and F(iv) yield the succession  $(1,2,2,3)\mapsto(0,4,1,3)\mapsto(1,2,5,0)\mapsto(1,1,4,2)$, where the last index is in the special cycle (\ref{scyc8}). If $k_1=1$, then Theorem \ref{ksucc}, cases B(iv), G(vi) and C(i) yield the succession  $(1,1,1,2)\mapsto(0,3,0,2)\mapsto(0,2,2,1)\mapsto(1,0,3,1)$, where the last index is in the special cycle (\ref{scyc5}).

Next suppose that $k_2=k_1+3$, so $\mathbf{k}=(1,k_1,k_1+3,k_1+1)$ in case B(iv). If $k_1=1$, then $\mathbf{k}=(1,1,4,2)$, in the special cycle (\ref{scyc8}). If $k_1=2$, then Theorem \ref{ksucc}, cases B(iv) and C(vii) yield the succession  $(1,2,5,3)\mapsto(0,4,1,6)\mapsto(1,3,6,1)$, where the last index is in the single-parameter cycle (\ref{regcyc}) with $t=1$. If $k_1=3$, then Theorem \ref{ksucc}, cases B(iv), C(vi) and F(iv) yield the succession  $(1,3,6,4)\mapsto(0,5,2,7)\mapsto(1,5,8,0)\mapsto(1,4,7,2)$, where the last index is in the single-parameter cycle (\ref{regcyc}) with $t=2$.  If $k_1=4$, then Theorem \ref{ksucc}, cases B(iv), C(v) and A(iii) yield the succession  $(1,4,7,5)\mapsto(0,6,3,8)\mapsto(1,6,9,1)\mapsto(1,6,6,4)$, where the last index is an almost-symmetric fixed point. If $k_1\ge5$, then Theorem \ref{ksucc}, cases B(iv), C(v) and B(i) yield the succession  $(1,k_1,k_1+3,k_1+1)\mapsto(0,k_1+2,k_1-1,k_1+4)\mapsto(1,k_1+2,k_1+5,k_1-3)\mapsto(1,k_1+2,k_1+2,k_1)$, where the last index is an almost-symmetric fixed point.

Finally, suppose that $k_2=k_1-3$, so $\mathbf{k}=(1,k_1,k_1-3,k_1+1)$ in case B(iv), and note that we then require $k_1\ge4$. If $k_1=4$, then Theorem \ref{ksucc}, cases B(iv) and C(iii) yield the succession $(1,4,1,5)\mapsto(0,6,3,2)\mapsto(1,3,6,1)$, where the last index is in the single-parameter cycle (\ref{regcyc}) with $t=1$. If $k_1=5$, then Theorem \ref{ksucc}, cases B(iv), C(ii) and F(iv) yield the succession  $(1,5,2,6)\mapsto(0,7,4,3)\mapsto(1,5,8,0)\mapsto(1,4,7,2)$, where the last index is in the single-parameter cycle (\ref{regcyc}) with $t=2$. If $k_1=6$, then Theorem \ref{ksucc}, cases B(iv), C(i) and A(iii) yield the succession  $(1,6,3,7)\mapsto(0,8,5,4)\mapsto(1,6,9,1)\mapsto(1,6,6,4)$, where the last index is an almost-symmetric fixed point. If $k_1\ge7$, then Theorem \ref{ksucc}, cases B(iv), C(i) and B(i) yield the succession  $(1,k_1,k_1-3,k_1+1)\mapsto(0,k_1+2,k_1-1,k_1-2)\mapsto(1,k_1,k_1+3,k_1-5)\mapsto(1,k_1,k_1,k_1-2)$, where the last index is an almost-symmetric fixed point.
\end{proof}

\begin{theorem}
The only fixed points and cycles in base $4$ are those listed in Subsection \ref{b4fpc}.
\end{theorem}
\begin{proof}
Theorems \ref{fpb4}, \ref{nofull} and \ref{b4biv} identify all full fixed points and all cycles containing a full Kaprekar index in base $4$, and also that the only non-full fixed points are triads. Thus it only remains to check for the existence of cycles consisting entirely of non-full indices, which will be done by determining the succession from all non-full indices, i.e., classes C--G in Theorem \ref{ksucc}. As we progress through an exhaustive consideration of cases in these classes, any case in which the succession leads to an index in a case already considered needs no further analysis. In particular, this applies if we encounter a full index; similarly, while considering the succession from cases in Classes D--G, we will frequently encounter an index of Class C, from which the succession will have already been determined; so the analysis of the case can then be terminated.
\begin{description}
\item[Case C(i).]  From the conditions $k_3<k_1<k_2+k_3$, we have that $k_1\ge2$ and so $k_2+k_3\ge3$. $\mathbf{k}'$ is full unless $k_1=2$, in which case $\mathbf{k}'=(1,0,3,k_2+k_3-2)$, in Class D. If $k_2+k_3=3$, then $\mathbf{k}'=(1,0,3,1)$, in the special cycle (\ref{scyc5}); while if $k_2+k_3>3$, then Theorem \ref{ksucc}, cases D(ii)--(iv) yield the next successor $\mathbf{k}''$ as a full index.
\item[Case C(ii).] From the condition $k_1=k_2+k_3$, we have that $k_1\ge2$. If $k_1=2$, then $\mathbf{k}'=(1,0,3,0)$, in the special cycle (\ref{scyc4}). If $k_1>2$, then $\mathbf{k}'=(1,k_1-2,k_1+1,0)$, which is case F(iv), with successor \mbox{$\mathbf{k}''=(1,k_1-3,k_1,2)$}. If $k_1=3$, this is $\mathbf{k}''=(1,0,3,2)$, which is case D(ii), with successor $\mathbf{k}'''=(1,1,1,3)$, which is full; while if $k_1>3$, then $\mathbf{k}''$ is full.
\item[Case C(iii).] In this case, $\mathbf{k}'$ is full unless $k_2+k_3=2$, in which case $\mathbf{k}'=(1,0,3,k_1-2)$ where, from the conditions, $k_1\ge3$. Indices of this form have already been considered in the succession from case C(i).
\item[Case C(iv).] If $k_2=k_3$, this is a triad fixed point. If $k_2<k_3$, then $\mathbf{k}'$ is case C(i), see above. If $k_2>k_3$, then $\mathbf{k}'$ is case C(v), see below.
\item[Cases C(v)--(vii).] In these cases $\mathbf{k}'$ has the same form as in the respective cases C(i)--(iii), with $k_1$ and $k_3$ exchanged. So the possible successions from $\mathbf{k}'$ are the same as in those cases.
\item[Case D(i).] In this case, $\mathbf{k}'$ is full unless $k_0=1$, in which case $\mathbf{k}'$ is Class C.
\item[Cases D(ii)-(iv).] In all these cases, $\mathbf{k}'$ is full.
\item[Case E(i).] In this case, $\mathbf{k}'$ is full unless $k_0=1$, in which case $\mathbf{k}'$ is Class C.
\item[Case E(ii).] In this case, $\mathbf{k}'$ is full unless $k_0=1$ or $k_3-k_0=1$. 

If $k_0=1$ but $k_3-k_0>1$, then $\mathbf{k}'$ is Class C. 

If $k_0>1$ and $k_3-k_0=1$, then $\mathbf{k}'=(k_0-1,3,0,k_0+k_1-1)$. 
\begin{itemize}
\item If $k_1\ge4$, this $\mathbf{k}'$ is case E(iv), with successor $\mathbf{k}''=(k_0-2,5,2,k_0+k_1-4)$, which is either full or Class C. 
\item If $k_1=3$, then $\mathbf{k}'$ is case E(iii), with successor $\mathbf{k}''=(k_0-2,5,2,k_0-1)$, which is either full or Class C. 
\item If $k_1=2$, then $\mathbf{k}'$ is case E(ii), with successor $\mathbf{k}''=(k_0-2,4,1,k_0)$, which is  either full or Class C.  
\item If $k_1=1$, then $\mathbf{k}'$ is case E(ii), with successor $\mathbf{k}''=(k_0-2,3,0,k_0+1)$; then if \mbox{$k_0=2$}, we have $\mathbf{k}''=(0,3,0,3)$, which is case G(v), with successor \mbox{$\mathbf{k}'''=(0,3,3,0)$}, case G(viii), with next successor $\mathbf{k}''''=(1,2,2,1)$, a symmetric fixed point; but if $k_0\ge3$, then $\mathbf{k}''$ is case E(iii), with successor $\mathbf{k}'''=(k_0-3,5,2,k_0-2)$, which is either full or Class C.
\end{itemize}

If both $k_0=1$ and $k_3-k_0=1$, then $\mathbf{k}'=(0,3,0,k_1)$. If $k_1=1$ or $k_1=2$, this $\mathbf{k}'$ is case G(vi), with successor $\mathbf{k}''=(0,k_1,k_1,3-k_1)$, which is Class C (and is in special cycle (\ref{scyc4}) if $k_1=1$). If $k_1=3$, then $\mathbf{k}'=(0,3,0,3)$, from which the succession has been determined above. If $k_1\ge4$, then $\mathbf{k}'$ is case G(iv), with successor $\mathbf{k}''=(0,3,3,k_1-3)$, which is Class C.
\item[Case E(iii).] In this case, $\mathbf{k}'$ is full or Class C unless $k_1=1$, when \mbox{$\mathbf{k}''=(k_0-1,3,0,k_0)$}. If $k_0=1$, then $\mathbf{k}''=(0,3,0,1)$, from which the succession has been determined in the analysis of case E(ii) above. If $k_0\ge2$, then $\mathbf{k}''$ itself is case E(ii).
\item[Case E(iv).] In this case, $\mathbf{k}'$ is full or Class C unless $k_1=1$, when \mbox{$\mathbf{k}''=(k_0-1,3,0,k_3-1)$}. But $k_3>1$, so the possibilities here are the same as in the analysis of case E(iii).
\item[Class F.] The successors $\mathbf{k}'$ in cases F(i)--(iv) are identical to those in cases C(iv)-(viii), respectively, with $k_1,k_2,k_3$ in Class C replaced by $k_0,k_1,k_2$ in Class F. So the further successions are as found for Class C.
\item[Case G(i).] In this case, $\mathbf{k}'$ is full unless $k_2=1$, when $\mathbf{k}'=(1,0,0,k_3)$. Since $k_3\ge2$ due to the conditions $k_3>k_2>0$, this $\mathbf{k}'$ is case G(x), with successor $\mathbf{k}''=(0,1,1,k_3-1)$, which is Class C.
\item[Case G(ii).] In this case, $\mathbf{k}'$ is a symmetric fixed point unless $k_2=1$. If $k_2=1$, then $\mathbf{k}'=(1,0,0,1)$, in the special cycle (\ref{scyc2}).
\item[Case G(iii).] In this case, $\mathbf{k}'$ has the same form as in case G(i), with $k_2$ and $k_3$ exchanged; so the further succession is the same as in case G(i).
\item[Case G(iv).] In this case, $\mathbf{k}'$ is Class C.
\item[Case G(v).] If $k_1=1$, then $\mathbf{k}$ is in the special cycle (\ref{scyc2}). If $k_1\ge2$, then $\mathbf{k}'$ is case G(viii), and the next successor $\mathbf{k}''$ is a symmetric fixed point.
\item[Case G(vi).] In this case, $\mathbf{k}'$ is Class C.
\item[Cases G(vii)--(ix).] In these cases, $\mathbf{k}'$ has the same form as in the respective case G(i)--(iii), with $k_2$ and $k_3$ replaced respectively with $k_1$ and $k_2$. So the further successions are as found in the earlier cases.
\item[Case G(x).] In this case, $\mathbf{k}'$ is either full or Class C.
\item[Case G(xi).] If $k_0=1$, then $\mathbf{k}$ is in the special cycle (\ref{scyc2}). If $k_0\ge2$, then $\mathbf{k}'$ is a symmetric fixed point.
\item[Cases G(xii) and (xiii).] In these cases, $\mathbf{k}'$ has the same form as in the respective case G(iv) and (v), with $k_1$ and $k_3$ replaced respectively with $k_0$ and $k_2$. So the further successions are as found in the earlier cases.
\item[Cases G(xiv) and (xv).] In these cases, $\mathbf{k}'$ has the same form as in the respective case G(i) and (ii), with $k_2$ and $k_3$ replaced respectively with $k_0$ and $k_1$. So the further successions are as found in the earlier cases.
\end{description}
\end{proof}

\section{Fixed points and cycles in base 6}\label{sec4}

We first list the fixed points and cycles that exist in base $6$, according to the classification in Section \ref{sec2}, and then we enumerate each of the classes of fixed points and cycles. We have not proved rigorously that our list is complete, but we provide some pointers as to how such a proof might be attempted.

\subsection{List of fixed points and cycles}

For base $6$, sequence \seqnum{A165060} in OEIS gives fixed points and least members of cycles (from which the remaining members of cycles may be computed), with sequence \seqnum{A165061} giving the lengths of cycles. From these data, we have classified the fixed points and cycles as follows.
\begin{description}
\item[Symmetric fixed points.] For every $k_0,k_1\in\nn$, the integer 
$$\underbrace{5\cdots5}_{k_0}\;\underbrace{3\cdots3}_{k_1}\;\underbrace{1\cdots1}_{k_1-1}0\underbrace{4\cdots4}_{k_1}\;\underbrace{2\cdots2}_{k_1}\;\underbrace{0\cdots0}_{k_0-1}1,$$
with Kaprekar index 
$$(k_0,k_1,k_1,k_1,k_1,k_0),$$
is a fixed point of the Kaprekar transformation, with digit-count $n=2k_0+4k_1$; so symmetric fixed points exist with even digit-counts $n\ge6$.
\item[Almost-symmetric fixed points.] Given $k_0,k_1,k_5$ with $1\le k_0<k_5<k_0+k_1$ and \mbox{$\alpha=k_5-k_0$}, the integer
$$\underbrace{5\cdots5}_{k_0}\;\underbrace{4\cdots4}_{\alpha}\underbrace{3\cdots3}_{k_1-\alpha}\;\underbrace{2\cdots2}_{\alpha}\;\underbrace{1\cdots1}_{k_1-\alpha-1}\;0\;\underbrace{5\cdots5}_{\alpha}\;\underbrace{4\cdots4}_{k_1-\alpha}\;\underbrace{3\cdots3}_{\alpha}\;\underbrace{2\cdots2}_{k_1-\alpha}\;\underbrace{1\cdots1}_{\alpha}\;\underbrace{0\cdots0}_{k_0-1}\;1,$$
with Kaprekar index
$$(k_0,k_1,k_1,k_1,k_1,k_5),$$
is an almost-symmetric fixed point with digit-count $n=k_0+4k_1+k_5$. Since the conditions require $k_0\ge1$, $k_5\ge2$, and $k_0+k_1\ge k_5+1\ge k_0+2$ so that $k_1\ge2$, the minimum digit-count for an almost-symmetric fixed point is $n=11$.
\item[Uniform zero-free fixed points.] For each $t\in\nn$, the integer
$$\underbrace{4\cdots4}_{t}\;\underbrace{2\cdots2}_{t-1}\;1\underbrace{5\cdots5}_{t}\;\underbrace{3\cdots3}_{t}\;\underbrace{1\cdots1}_{t-1}2,$$
with Kaprekar index 
$$(0,t,t,t,t,t),$$
is a fixed point of the Kaprekar transformation, with digit-count $n=5t$.
\item[Triad fixed points.] For each $t\in\nn$, the integer
$$\underbrace{3\cdots3}_{t-1}\;2\;\underbrace{5\cdots5}_{t}\;\underbrace{2\cdots2}_{t-1}\;3,$$
with Kaprekar index $(0,0,t,t,0,t)$, is a fixed point with digit-count $n=3t$.
\item[Other zero-free fixed points.] For each $t\in\nn$ and $u\in\nn$ with $u<t$, the integer
$$\underbrace{4\cdots4}_{u}\;\underbrace{3\cdots3}_{t-u}\;\underbrace{2\cdots2}_{u-1}\;1\underbrace{5\cdots5}_{t}\;\underbrace{3\cdots3}_{u}\;\underbrace{2\cdots2}_{t-u}\;\underbrace{1\cdots1}_{u-1}2,$$
with Kaprekar index 
$$(0,u,t,t,u,t),$$
is a fixed point of the Kaprekar transformation, with digit-count $n=3t+2u$. 

[Note: if we had specified $u\le t$ rather than $u<t$, our definition of ``other zero-free fixed points'' would also have embraced the uniform zero-free fixed points; but similar generalisations in higher bases are more complicated, which is why we have retained the distinction between the two classes of fixed points in base $6$. Setting $u=0$ yields the Kaprekar index of triad fixed points, but not the actual integer that is fixed under the Kaprekar transformation.]
\item[Single-parameter fixed points.] For each $t\in\nn$, the integer
$$4\underbrace{3\cdots3}_{t-1}\;2\underbrace{1\cdots1}_{t-1}\;0\underbrace{4\cdots4}_{t}\;3\underbrace{2\cdots2}_{t},$$
with Kaprekar index 
\beq\label{b6spfp}
(1,t-1,t+1,t,t+1,0),
\eeq
is a fixed point of the Kaprekar transformation, with digit-count $n=4t+2$. 
\end{description}
\begin{description}
\item[Symmetric and almost-symmetric cycles.] Since $\sigma(5)=2$, these cycles have length \mbox{$l=2$}. Given $k_0,k_5,t,u\in\nn$ with $u\ne t$ and either $k_0=k_5$ or $k_0< k_5<k_0+\min\{t,u\}$, cycles with digit-count $n=k_0+k_5+2(t+u)$ have the sequence of Kaprekar indices,
\beq\label{b6sasc}
(k_0,t,u,u,t,k_5)\mapsto(k_0,u,t,t,u,k_5)\mapsto(k_0,t,u,u,t,k_5).
\eeq
The integers represented by these Kaprekar indices can be obtained by the usual subtraction process (\ref{defkt}). 

Allowing $u=t$ would yield an [almost-]symmetric fixed point, rather than a cycle.
\item[Non-symmetric $\sigma$-cycles.] For each $t,u\in\nn$ with $u\ne t+1$, there is a cycle of length \mbox{$l=\sigma(5)=2$} and digit-count $n=2(t+u)+4$, with the sequence of Kaprekar indices,
\beq\label{nssigma6}
(1,t,u+1,u,t+2,0)\mapsto(1,u-1,t+2,t+1,u+1,0)\mapsto(1,t,u+1,u,t+2,0);
\eeq
so the rules of form (\ref{kdel}) for this cycle are:
$$k_2'=k_1+2,k_1'=k_2-2,k_4'=k_3+1,k_3'=k_4-1.$$
Allowing $u=t+1$ would yield a single-parameter fixed point with Kaprekar index of form (\ref{b6spfp}) after a change of notation, $t\to t-1$.
\item[Single-parameter cycles] For each $t\in\nn$, there is a cycle of length $l=2$ and digit-count $n=2t+5$, with the sequence of Kaprekar indices,
$$(1,t,1,0,t+2,1)\mapsto(0,2,t,t+1,0,2)\mapsto(1,t,1,0,t+2,1).$$
\item[Special cycles.] The following special cycles exist in base $6$. We display the cycles of both the integers and their Kaprekar indices:
\begin{itemize}
\item With $n=2$, a cycle of length $l=3$:
\begin{align*}&05\mapsto41\mapsto23\mapsto05;\\
&\quad(1,0,0,0,0,1)\mapsto(0,1,0,0,1,0)\mapsto(0,0,1,1,0,0)\mapsto(1,0,0,0,0,1).
\end{align*}
\item With $n=4$, a cycle of length $l=6$:
\begin{align*}&1554\mapsto4042\mapsto4132\mapsto3043\mapsto3552\mapsto3133\mapsto1554;\\
&\quad(0,1,0,0,1,2)\mapsto(1,0,1,0,2,0)\mapsto(0,1,1,1,1,0)\mapsto(1,0,0,2,1,0)\mapsto\\
&\qquad(0,0,1,1,0,2)\mapsto(0,1,0,3,0,0)\mapsto(0,1,0,0,1,2).
\end{align*}
\item With $n=5$, a cycle of length $l=2$:
\begin{align*}&31533\mapsto35552\mapsto31533;\\
&\quad(0,1,0,3,0,1)\mapsto(0,0,1,1,0,3)\mapsto(0,1,0,3,0,1).
\end{align*}
\item With $n=6$, a cycle of length $l=3$:
\begin{align*}&205544\mapsto525521\mapsto432222\mapsto205544;\\
&\quad(1,0,1,0,2,2)\mapsto(0,1,2,0,0,3)\mapsto(0,0,4,1,1,0)\mapsto(1,0,1,0,2,2).
\end{align*}
\item With $n=8$, a cycle of length $l=7$:
\begin{align*}&31104443\mapsto43255222\mapsto33204323\mapsto41055442\mapsto54155311\mapsto44404112\mapsto\\
&\quad43313222\mapsto31104443;\\
&\quad(1,2,0,2,3,0)\mapsto(0,0,4,1,1,2)\mapsto(1,0,2,4,1,0)\mapsto(1,1,1,0,3,2)\mapsto\\
&\qquad(0,3,0,1,1,3)\mapsto(1,2,1,0,4,0)\mapsto(0,1,3,3,1,0)\mapsto(1,2,0,2,3,0).
\end{align*}
\end{itemize}
\end{description}

We have not attempted a rigorous proof that our list of fixed points and cycles is complete. However, any such attempt should make use of the restrictions on fixed points and cycle members implied by Lemma \ref{divb}(b) and Theorem \ref{noex}(b). In base $6$ the former yields that in any succession, $k_1'+2k_2'+3k_3'+4k_4'$ is divisible by $5$, and hence that $(k_1'-k_4')+2(k_2'-k_3')$ is divisible by $5$. Then applying Theorem \ref{noex}(b), we find that there are only four possibilities for the Kaprekar index component differences in any fixed point or cycle member:
\begin{enumerate}
\item $k_1-k_4=0$ and $k_2-k_3=0$;
\item $k_1-k_4=\pm1$ and $k_2-k_3=\pm2$;
\item $k_1-k_4=\pm1$ and $k_2-k_3=\mp3$;
\item $k_1-k_4=\pm2$ and $k_2-k_3=\mp1$;
\end{enumerate}
Option 1 is found in all members of the symmetric and almost-symmetric cycles, and in all the fixed points except for the single-parameter fixed points, which take option 4. Option 4 is also found in all members of the non-symmetric $\sigma$-cycles and the single-parameter cycles. All options appear among the members of the five special cycles. 

Any proof that our list of fixed points in cycles is complete would need to consider successions from all four options for the differences $k_1-k_4$ and $k_2-k_3$, and would also need to take into account all possibilities for $k_0$ and $k_5$, for which the only restrictions are that either $k_5\ge k_0$ or $k_5=0$ with $k_0=1$. Hence the number of cases to be considered would be even greater than in the proof by exhaustion for base $4$.

\subsection{Enumeration of fixed points and cycles}

\begin{description}
\item[Symmetric and almost-symmetric fixed points.] In base $b=6$, the formula (\ref{ensfp}) for the number of symmetric fixed points for any even $n$ becomes
\beq\label{scount6}
N_S=\left\lfloor\frac{n-2}{4}\right\rfloor
\eeq 
(and $N_S=0$ for odd $n$).

The formulae (\ref{ascountogen}) and (\ref{ascountegen}) for the numbers of almost-symmetric fixed points with odd and even digit-counts become
\beq\label{ascounto6}
N_A=\sum_{m=1}^{\lfloor(n-1)/10\rfloor}\left\lfloor\frac{n-10m+3}{4}\right\rfloor\quad\text{for odd $n$}
\eeq
and
\beq\label{ascounte6}
N_A=\sum_{m=1}^{\lfloor(n-6)/10\rfloor}\left\lfloor\frac{n-10m-2}{4}\right\rfloor\quad\text{for even $n$}.
\eeq
Closed-form expressions for these sums depend on congruence of $n$ modulo $20$, but we can reduce the amount of calculation required by noting equations (\ref{asevenodd}) and (\ref{sasevenodd}), which in base $6$ can be combined in the form
\beq\label{nansa6}
N_A(n+10)=N_A(n+5)=N_{SA}(n)\quad\text{for even $n$},
\eeq
where $N_{SA}=N_S+N_A$. Thus it is only necessary to find a formula for $N_{SA}$ for even $n$, and we then have $N_A$ and $N_{SA}$ for all even and odd $n$ (noting that $N_A=0$ for odd $n\le9$ and for even $n\le14$, and that $N_{SA}=N_A$ when $n$ is odd). Evaluating (\ref{scount6}) and (\ref{ascounte6}) and adding, we find that all cases are covered by the formula,
$$N_{SA}=\left\lfloor\frac{(n+2)^2}{80}+\frac{1}{5}\right\rfloor\quad\text{for even $n$}.$$
Results of a numerical evaluation of this formula for $6\le n\le40$ are tabulated below (note that $N_{SA}=0$ for $n<6$), with all non-zero values of $N_A$ then following from (\ref{nansa6}).

\begin{center}
\begin{tabular}{|c|c||c|c||c|c||c|c||c|c|}
\hline
$n$&$N_{SA}$&$n$&$N_{SA}$&$n$&$N_{SA}$&$n$&$N_{SA}$&$n$&$N_{SA}$\\
\hline
6&1&14&3&22&7&30&13&38&20\\
8&1&16&4&24&8&32&14&40&22\\
10&2&18&5&26&10&34&16&&\\
12&2&20&6&28&11&36&18&&\\
\hline
\end{tabular}
\end{center}

\item[Zero-free fixed points.]
To enumerate all zero-free fixed points (uniform, triad and others) with a given digit-count $n$, we need to enumerate pairs $(t,u)$ of non-negative integers with $u\le t$ and $3t+2u=n$. This is equivalent to enumerating all pairs $(t,v)$ of non-negative integers with $3t+5v=n$. Since $3$ is coprime with $5$, this yields $N_Z$, the number of zero-free fixed points when $n\equiv r$ (mod $15$), as
$$N_Z=\left\lfloor\frac{n}{15}\right\rfloor+q,$$
where $q=0$ for $r=1,2,4,7$, and $q=1$ for all other non-negative $r\le14$.
\item[Single-parameter fixed points.]
There is just one single-parameter fixed point whenever $n\equiv2$ (mod $4$), excluding $n=2$, and none for any other value of $n$.
\item[Symmetric and almost-symmetric cycles.]

Referring to the notation used in Subsection \ref{enumsasc}, we have in base $6$ that $B=5$ and $C=2$; both $B$ and $C$ are prime, with $\sigma(B)=C$, so formulae in Subsection \ref{enumsasc} can be used without any additional computations. From (\ref{ngamt}) and (\ref{ngamma}) we obtain that
\beq\label{ngam6}
N_\gamma(2,\omega)=\left\lfloor\frac{\omega-1}{2}\right\rfloor.
\eeq
The summation in (\ref{neta}) then needs to be evaluated separately for the cases of even and odd $\eta$, but the results from the two cases can be combined neatly in the formula,
\beq\label{seta}
N_{k\gamma}(2,\eta)=\left\lfloor\frac{(\eta-2)^2}{4}\right\rfloor.
\eeq
The number of symmetric cycles with digit-count $n$ is then
\beq\label{nsc6count}
N_{SC}=N_{k\gamma}\left(2,\frac{n}{2}\right)=\left\lfloor\frac{(n-4)^2}{16}\right\rfloor
\eeq
for even $n$ (and $N_{SC}=0$ for odd $n$). The numbers of almost-symmetric cycles with odd and even $n$ are obtained respectively from (\ref{nacodd}) and (\ref{naceven}) as
\beq\label{nacodd6count}
N_{AC}=\sum_{m=1}^{\lfloor(n-3)/10\rfloor}\left\lfloor\frac{(n-10m+1)^2}{16}\right\rfloor
\eeq
and
\beq\label{naceven6count}
N_{AC}=\sum_{m=1}^{\lfloor(n-8)/10\rfloor}\left\lfloor\frac{(n-10m-4)^2}{16}\right\rfloor.
\eeq
As with the symmetric and almost-symmetric fixed points, closed-form expressions for the above sums depend on congruence of $n$ modulo $20$, but we can reduce the amount of calculation needed by using equation (\ref{ascevenodd}), which takes the form
\beq\label{nansac6}
N_{AC}(n+10)=N_{AC}(n+5)=N_{SAC}(n)\quad\text{for even $n$}
\eeq
in base $6$. Evaluating (\ref{nsc6count}) and (\ref{naceven6count}) and adding, we find that all even values of $n$ are covered by the formula,
\beq\label{nsac6}
N_{SAC}=\left\lfloor\frac{n(n-4)(n+7)}{480}+\frac{1}{5}\right\rfloor\quad\text{for even $n$},
\eeq
from which $N_A$ may be found using (\ref{nansac6}) in every case where $N_A$ is non-zero. Results of a numerical evaluation of this formula for $8\le n\le40$ are tabulated below (note that $N_{SAC}=0$ for $n<8$).

\begin{center}
\begin{tabular}{|c|c||c|c||c|c||c|c||c|c|}
\hline
$n$&$N_{SAC}$&$n$&$N_{SAC}$&$n$&$N_{SAC}$&$n$&$N_{SAC}$&$n$&$N_{SAC}$\\
\hline
8&1&16&9&24&31&32&73&40&141\\
10&2&18&13&26&39&34&87&&\\
12&4&20&18&28&49&36&103&&\\
14&6&22&24&30&60&38&121&&\\
\hline
\end{tabular}
\end{center}

\item[Non-symmetric $\sigma$-cycles.] From the specification of these cycles in (\ref{nssigma6}), $k_1\ge0$ and \mbox{$k_2\ge2$}; so, since $0+2=1+1$, the number of pairs of integers satisfying these requirements and adding to a given $\omega$ is the same as when we require $k_1\ge1$ and $k_2\ge1$ (as in the case of symmetric cycles). Now, $n=2(k_1+k_2)+2$ from (\ref{nssigma6}), so $\omega:=k_1+k_2=(n-2)/2$. Hence the number of non-symmetric $\sigma$-cycles is given by (\ref{ngam6}) as
$$N_{N\sigma}=N_\gamma\left(2,\frac{n-2}{2}\right)=\left\lfloor\frac{n-4}{4}\right\rfloor$$
for even $n$, and zero for odd $n$.
\item[Single-parameter cycles]
There is one single-parameter cycle for each odd $n\ge7$.
\item[Special cycles]
There is one special cycle for each of the digit-counts $n=2,4,5,6,8$.
\end{description}
Finally, we note that the following fixed points and cycles are unanimous: the triad fixed point with $n=3$, the special cycles with $n=2$ and $n=4$, and the single-parameter cycle with $n=7$.

\section{Fixed points and cycles in base 8}\label{sec4}

For base $8$, sequence \seqnum{A165099} in OEIS gives fixed points and least members of cycles (from which the remaining members of cycles may be computed), with sequence \seqnum{A165100} giving the lengths of cycles. From these data, we have classified the fixed points and cycles as follows. We have not proved rigorously that this classification is complete.

\subsection{List of fixed points and cycles}\label{lfpc8}

\begin{description}
\item[Symmetric fixed points.] For every $k_0,k_1\in\nn$, the integer 
$$\underbrace{7\cdots7}_{k_0}\;\underbrace{5\cdots5}_{k_1}\;\underbrace{3\cdots3}_{k_1}\;\underbrace{1\cdots1}_{k_1-1}0\underbrace{6\cdots6}_{k_1}\;\underbrace{4\cdots4}_{k_1}\;\underbrace{2\cdots2}_{k_1}\;\underbrace{0\cdots0}_{k_0-1}1,$$
with Kaprekar index 
$$(k_0,k_1,k_1,k_1,k_1,k_1,k_1,k_0),$$
is a fixed point of the Kaprekar transformation, with digit-count $n=2k_0+6k_1$, so symmetric fixed points exist with even digit-counts $n\ge8$.
\item[Almost-symmetric fixed points.] Given $k_0,k_1,k_7$ with $1\le k_0<k_7<k_0+k_1$ and \mbox{$\alpha=k_7-k_0$}, the integer
\begin{align*}&\underbrace{7\cdots7}_{k_0}\;\underbrace{6\cdots6}_{\alpha}\;\underbrace{5\cdots5}_{k_1-\alpha}\;\underbrace{4\cdots4}_{\alpha}\underbrace{3\cdots3}_{k_1-\alpha}\;\underbrace{2\cdots2}_{\alpha}\;\underbrace{1\cdots1}_{k_1-\alpha-1}\;0\;\underbrace{7\cdots7}_{\alpha}\;\underbrace{6\cdots6}_{k_1-\alpha}\;\underbrace{5\cdots5}_{\alpha}\;\underbrace{4\cdots4}_{k_1-\alpha}\\
&\qquad\qquad\qquad\qquad\qquad\qquad\qquad\qquad\qquad\qquad\qquad\qquad\underbrace{3\cdots3}_{\alpha}\;\underbrace{2\cdots2}_{k_1-\alpha}\;\underbrace{1\cdots1}_{\alpha}\;\underbrace{0\cdots0}_{k_0-1}\;1,
\end{align*}
with Kaprekar index
$$(k_0,k_1,k_1,k_1,k_1,k_1,k_1,k_7),$$
is an almost-symmetric fixed point with digit-count $n=k_0+6k_1+k_7$. Since the conditions require $k_0\ge1$, $k_7\ge2$, and $k_1\ge2$, the minimum digit-count for an almost-symmetric fixed point is $n=15$.
\item[Uniform zero-free fixed points.] For each $t\in\nn$, the integer
$$\underbrace{6\cdots6}_{t}\;\underbrace{4\cdots4}_{t}\;\underbrace{2\cdots2}_{t-1}\;1\underbrace{7\cdots7}_{t}\;\underbrace{5\cdots5}_{t}\;\underbrace{3\cdots3}_{t}\;\underbrace{1\cdots1}_{t-1}2,$$
with Kaprekar index 
$$(0,t,t,t,t,t,t,t),$$
is a fixed point of the Kaprekar transformation, with digit-count $n=7t$.
\item[Triad fixed points.] For each $t\in\nn$, the integer
$$\underbrace{4\cdots4}_{t-1}\;3\;\underbrace{7\cdots7}_{t}\;\underbrace{3\cdots3}_{t-1}\;4,$$
with Kaprekar index $(0,0,0,t,t,0,0,t)$, is a fixed point with digit-count $n=3t$.
\item[Other zero-free fixed points.]  There are two further classes of zero-free fixed points in base $8$.
\begin{itemize}
\item[(a)] For each $t\in\nn$ and $u\in\nn$ with $u<t$, the integer
$$\underbrace{6\cdots6}_{t}\;\underbrace{4\cdots4}_{u}\;\underbrace{3\cdots3}_{t-u}\;\underbrace{2\cdots2}_{u-1}\;1\underbrace{7\cdots7}_{t}\;\underbrace{5\cdots5}_{u}\;\underbrace{4\cdots4}_{t-u}\;\underbrace{3\cdots3}_{u}\;\underbrace{1\cdots1}_{t-1}2,$$
with Kaprekar index 
$$(0,t,u,t,t,u,t,t),$$
is a fixed point of the Kaprekar transformation, with digit-count $n=5t+2u$. As in base $6$, allowing $u=t$ in the definition of these fixed points would yield the uniform zero-free fixed points. 
\item[(b)] For each $t\in\nn$ and $\epsilon\in\nn$ with $\epsilon<t$, the integer
$$\underbrace{6\cdots6}_{t-\epsilon}\;\underbrace{5\cdots5}_{\epsilon}\;\underbrace{4\cdots4}_{t-\epsilon}\;\underbrace{3\cdots3}_{\epsilon}\;\underbrace{2\cdots2}_{t-1}\;1\underbrace{7\cdots7}_{t}\;\underbrace{5\cdots5}_{t}\;\underbrace{4\cdots4}_{\epsilon}\;\underbrace{3\cdots3}_{t-\epsilon}\;\underbrace{2\cdots2}_{\epsilon}\;\underbrace{1\cdots1}_{t-\epsilon-1}2,$$
with Kaprekar index 
$$(0,t-\epsilon,t+\epsilon,t,t,t+\epsilon,t-\epsilon,t),$$
is a fixed point of the Kaprekar transformation, with digit-count $n=7t$. Allowing $\epsilon=0$ in the definition of these fixed points would yield the uniform zero-free fixed points.
\end{itemize}
\item[Single-parameter fixed points.] For each $t\in\nn$, the integer
$$6\underbrace{5\cdots5}_{t-1}\;4\underbrace{3\cdots3}_{t-1}\;\underbrace{1\cdots1}_{t-1}\;0\underbrace{6\cdots6}_{t}\;\underbrace{4\cdots4}_{t-1}\;3\underbrace{2\cdots2}_{t},$$
with Kaprekar index 
\beq\label{b8spfp}
(1,t-1,t,t,t,t-1,t+1,0),
\eeq
is a fixed point of the Kaprekar transformation, with digit-count $n=6t$. These have some similarities to the single-parameter fixed points in base $6$, in particular having $k_0=1$ and $k_B=0$, but not sufficient similarities to indicate a general class of single-parameter fixed points for all even bases $b\ge6$.
\item[Special fixed points.]
There is just one special fixed point in base $8$: for $n=2$, the integer $25$, with Kaprekar index $(0,0,1,0,0,1,0,0)$, is a fixed point.
\end{description}
For cycles, we display only the Kaprekar indices of cycle members; the integers represented by these Kaprekar indices can be obtained by the usual subtraction process (\ref{defkt}). 
\begin{description}
\item[Symmetric and almost-symmetric cycles.] Since $\sigma(7)=3$, these cycles have length \mbox{$l=3$}. Given $k_0,k_7,t,u,v\in\nn$ with $t,u,v$ not all equal, $v=\min\{t,u,v\}$ and either $k_0=k_7$ or $k_0< k_7<k_0+v$, cycles with digit-count $n=k_0+k_7+2(t+u+v)$ have the sequence of Kaprekar indices,
\begin{align*}
(k_0,t,u,v,v,u,t,k_7)\mapsto(k_0,v,t,u,u,t,v,k_7)\mapsto&(k_0,u,v,t,t,v,u,k_7)\mapsto\\
&\qquad(k_0,t,u,v,v,u,t,k_7).
\end{align*}
The condition that $t,u,v$ are not all equal, with $v=\min\{t,u,v\}$, implies that $t+u+v\ge3v+1$. Thus for symmetric cycles, with $k_0\ge1,k_7\ge1$, and $v\ge1$, we have that $n\ge10$. For almost-symmetric cycles, with $k_0\ge1,k_7\ge2$, and $v\ge2$, we have that $n\ge17$.
\item[Non-symmetric $\sigma$-cycles.] For each $t,u,v\in\nn$ with $t+1,u,v$ not all equal, there is a cycle of length \mbox{$l=\sigma(7)=3$} and digit-count $n=2(t+u+v)+2$, with the sequence of Kaprekar indices,
\begin{align}
&(1,t,u,v,v,u-1,t+2,0)\mapsto(1,v-1,t+1,u,u,t,v+1,0)\mapsto\nonumber\\
&\qquad(1,u-1,v,t+1,t+1,v-1,u+1,0)\mapsto(1,t,u,v,v,u-1,t+2,0);\label{nssigma8}
\end{align}
so the rules of form (\ref{kdel}) for this cycle are:
$$k_2'=k_1+1,k_3'=k_2,k_1'=k_3-1,k_6'=k_4+1,k_4'=k_5+1,k_5'=k_6-2.$$
In the case where $t+1=u=v$, we have the single-parameter fixed points with Kaprekar index of form (\ref{b8spfp}) (after a change of notation, $t\to t-1$).
\item[Single-parameter cycles] There are several of these in base $8$, with a variety of lengths. In each case listed below, the cycle exists for every $t\in\nn$.

(i) With digit-count $n=7t+4$, cycles of length $l=2$:
\begin{align*}
(1,t,t+1,t,t,t,t+2,t)\mapsto&(0,t+1,t,t+1,t+1,t+1,t-1,t+1)\mapsto\\
&\qquad(1,t,t+1,t,t,t,t+2,t).
\end{align*}
(ii) With digit-count $n=5t+7$, cycles of length $l=2$:
\begin{align*}
(1,0,2t+2,0,0,2t+1,2,t+1)\mapsto&(0,1,t+1,t+1,t+2,t,0,t+2)\mapsto\\
&\qquad(1,0,2t+2,0,0,2t+1,2,t+1).
\end{align*}
(iii) With digit-count $n=2t+2$, cycles of length $l=3$:
\begin{align*}
(1,t-1,0,1,0,1,t,0)\mapsto&(0,1,t,0,0,t,1,0)\mapsto\\
&\qquad(0,0,1,t,t,1,0,0)\mapsto(1,t-1,0,1,0,1,t,0).
\end{align*}
(iv) With digit-count $n=7t+6$, cycles of length $l=5$:
\begin{align*}
&(1,t-1,t+2,t+1,t+1,t+1,t+1,t)\mapsto(1,t,t,t+2,t+2,t,t,t+1)\mapsto\\
&\qquad(1,t+2,t,t,t,t,t+2,t+1)\mapsto(0,t+2,t+1,t,t,t+2,t,t+1)\mapsto\\
&\qquad(0,t+1,t+1,t+1,t+1,t+1,t+1,t)\mapsto\\
&\qquad(1,t-1,t+2,t+1,t+1,t+1,t+1,t).
\end{align*}
(v) With digit-count $n=7t+10$, cycles of length $l=5$:
\begin{align*}
&(1,t+1,t+2,t+1,t+1,t+1,t+3,t)\mapsto(1,t,t+2,t+2,t+2,t+2,t,t+1)\mapsto\\
&\qquad(1,t+2,t,t+2,t+2,t,t+2,t+1)\mapsto(1,t+2,t+2,t,t,t+2,t+2,t+1)\mapsto\\
&\qquad(0,t+2,t+1,t+2,t+2,t+2,t,t+1)\mapsto\\
&\qquad(1,t+1,t+2,t+1,t+1,t+1,t+3,t).
\end{align*}
(vi) With digit-count $n=7t+4$, cycles of length $l=6$:
\begin{align*}
&(1,t-1,t+1,t+1,t+1,t,t+1,t)\mapsto(1,t,t,t+1,t+1,t,t,t+1)\mapsto\\
&\qquad(1,t+1,t,t,t,t,t+1,t+1)\mapsto(0,t+2,t,t,t,t+1,t,t+1)\mapsto\\
&\qquad(0,t+1,t+1,t,t,t+1,t+1,t)\mapsto(0,t,t+1,t+1,t+1,t+1,t,t)\mapsto\\
&\qquad(1,t-1,t+1,t+1,t+1,t,t+1,t).
\end{align*}
(vii) With digit-count $n=7t+6$, cycles of length $l=9$:
\begin{align*}
&(1,t+1,t+1,t,t,t,t+3,t)\mapsto(0,t+1,t+1,t+1,t+1,t+2,t-1,t+1)\mapsto\\
&\qquad(1,t,t+1,t+1,t+1,t,t+2,t)\mapsto(1,t,t+1,t+1,t+1,t+1,t,t+1)\mapsto\\
&\qquad(1,t+1,t,t+1,t+1,t,t+1,t+1)\mapsto(1,t+1,t+1,t,t,t+1,t+1,t+1)\mapsto\\
&\qquad(0,t+2,t,t+1,t+1,t+1,t,t+1)\mapsto(1,t,t+2,t,t,t+1,t+2,t)\mapsto\\
&\qquad(0,t+1,t,t+2,t+2,t+1,t-1,t+1)\mapsto(1,t+1,t+1,t,t,t,t+3,t).
\end{align*}
(viii) With digit-count $n=7t+8$, cycles of length $l=12$:
\begin{align*}
&(1,t+1,t+1,t+1,t+1,t,t+3,t)\mapsto(1,t,t+2,t+1,t+1,t+2,t,t+1)\mapsto\\
&\qquad(1,t+1,t,t+2,t+2,t,t+1,t+1)\mapsto(1,t+2,t+1,t,t,t+1,t+2,t+1)\mapsto\\
&\qquad(0,t+2,t+1,t+1,t+1,t+2,t,t+1)\mapsto(1,t,t+2,t+1,t+1,t+1,t+2,t)\mapsto\\
&\qquad(1,t,t+1,t+2,t+2,t+1,t,t+1)\mapsto(1,t+2,t,t+1,t+1,t,t+2,t+1)\mapsto\\
&\qquad(1,t+1,t+2,t,t,t+2,t+1,t+1)\mapsto(0,t+2,t,t+2,t+2,t+1,t,t+1)\mapsto\\
&\qquad(1,t+1,t+2,t,t,t+1,t+3,t)\mapsto(0,t+1,t+1,t+2,t+2,t+2,t-1,t+1)\mapsto\\
&\qquad(1,t+1,t+1,t+1,t+1,t,t+3,t).
\end{align*}
\item[Special cycles.] The following special cycles exist in base $8$. 
\begin{itemize}
\item With $n=4$, a cycle of length $l=5$:
\begin{align*}
&\quad(0,1,0,0,0,0,1,2)\mapsto(1,0,1,0,0,0,2,0)\mapsto(0,0,1,2,0,0,1,0)\mapsto\\
&\qquad(0,0,0,1,1,0,0,2)\mapsto(0,0,1,0,3,0,0,0)\mapsto(0,1,0,0,0,0,1,2).
\end{align*}
\item With $n=5$, a cycle of length $l=2$:
$$\quad(0,0,1,0,3,0,0,1)\mapsto(0,0,0,1,1,0,0,3)\mapsto(0,0,1,0,3,0,0,1).$$
\item With $n=5$, a cycle of length $l=4$:
\begin{align*}
&\quad(0,1,0,1,0,2,0,1)\mapsto(0,1,1,0,0,1,1,1)\mapsto(0,0,1,2,0,0,1,1)\mapsto\\
&\qquad(0,0,1,1,1,1,0,1)\mapsto(0,1,0,1,0,2,0,1).
\end{align*}
\item With $n=7$, a cycle of length $l=4$:
\begin{align*}
&\quad(1,0,1,1,1,0,2,1)\mapsto(0,1,1,1,2,0,0,2)\mapsto(1,0,2,0,0,1,2,1)\mapsto\\
&\qquad(0,1,0,3,1,0,0,2)\mapsto(1,0,1,1,1,0,2,1).
\end{align*}
\item With $n=7$, a cycle of length $l=7$:
\begin{align*}
&\quad(1,0,0,2,1,1,1,1)\mapsto(0,2,0,1,1,1,0,2)\mapsto(0,2,1,0,0,1,2,1)\mapsto\\
&\qquad(0,0,2,2,0,1,1,1)\mapsto(0,1,0,2,1,2,0,1)\mapsto(0,1,2,0,0,2,1,1)\mapsto\\
&\qquad(0,0,2,1,2,0,1,1)\mapsto(1,0,0,2,1,1,1,1).
\end{align*}
\item With $n=9$, a cycle of length $l=4$:
\begin{align*}
&\quad(1,1,2,0,0,1,3,1)\mapsto(0,1,1,3,1,1,0,2)\mapsto(1,0,3,0,0,2,2,1)\mapsto\\
&\qquad(0,1,1,2,3,0,0,2)\mapsto(1,1,2,0,0,1,3,1).
\end{align*}
\item With $n=9$, a cycle of length $l=5$:
\begin{align*}
&\quad(1,0,0,3,2,1,1,1)\mapsto(1,1,1,1,1,1,1,2)\mapsto(0,3,0,1,1,1,1,2)\mapsto\\
&\qquad(0,2,2,0,0,2,2,1)\mapsto(0,0,3,1,2,1,1,1)\mapsto(1,0,0,3,2,1,1,1).
\end{align*}
\item With $n=9$, a cycle of length $l=5$:
\begin{align*}
&\quad(1,0,1,2,2,0,2,1)\mapsto(1,1,0,2,2,0,1,2)\mapsto(1,2,0,1,1,0,2,2)\mapsto\\
&\qquad(0,2,2,0,1,1,1,2)\mapsto(0,1,2,1,2,0,2,1)\mapsto(1,0,1,2,2,0,2,1).
\end{align*}
\item With $n=11$, a cycle of length $l=4$:
\begin{align*}
&\quad(1,0,3,1,1,2,2,1)\mapsto(0,2,0,3,3,1,0,2)\mapsto(1,2,2,0,0,1,4,1)\mapsto\\
&\qquad(0,1,2,3,1,2,0,2)\mapsto(1,0,3,1,1,2,2,1).
\end{align*}
\item With $n=12$, a cycle of length $l=4$:
\begin{align*}
&\quad(1,0,2,2,2,1,2,2)\mapsto(0,2,1,2,2,2,0,3)\mapsto(0,2,3,0,0,3,2,2)\mapsto\\
&\qquad(0,1,2,2,3,0,2,2)\mapsto(1,0,2,2,2,1,2,2).
\end{align*}
\end{itemize}
Only the special fixed point with $n=2$ and the triad fixed point with $n=3$ are unanimous.
\end{description}

\subsection{Enumeration of fixed points and cycles}

\begin{description}
\item[Symmetric and almost-symmetric fixed points.] In base $b=8$, the formula (\ref{ensfp}) for the number of symmetric fixed points for any even $n$ becomes
\beq\label{nscount8}
N_S=\left\lfloor\frac{n-2}{6}\right\rfloor
\eeq 
(and $N_S=0$ for odd $n$).

The formulae (\ref{ascountogen}) and (\ref{ascountegen}) for the numbers of almost-symmetric fixed points with odd and even digit-counts become
\beq\label{ascounto8}
N_A=\sum_{m=1}^{\lfloor(n-1)/14\rfloor}\left\lfloor\frac{n-14m+5}{6}\right\rfloor\quad\text{for odd $n$}
\eeq
and
\beq\label{ascounte8}
N_A=\sum_{m=1}^{\lfloor(n-8)/14\rfloor}\left\lfloor\frac{n-14m-2}{6}\right\rfloor\quad\text{for even $n$}.
\eeq
Closed-form expressions for these sums depend on congruence of $n$ modulo $42$, but by evaluating (\ref{nscount8}) and (\ref{ascounte8}) and adding, we find that the total number of symmetric and almost-symmetric fixed points is given by the following formula for all even $n$:
\beq\label{nsa8}
N_{SA}:=N_S+N_A=\left\lfloor\frac{n(n+6)}{168}+\frac{1}{3}\right\rfloor.
\eeq
The relations (\ref{asevenodd}) and (\ref{sasevenodd}) become
\beq\label{nansa8}
N_A(n+14)=N_A(n+7)=N_{SA}(n)\quad\text{for even $n$}.
\eeq
Thus it is only necessary to find $N_{SA}$ for even $n$ in order to obtain $N_A$ and $N_{SA}$ for all even and odd $n$ (noting that $N_A=0$ for odd $n\le13$ and even $n\le20$). Values of $N_{SA}$ for even digit-counts $8\le n\le50$, evaluated from formula (\ref{nsa8}), are tabulated below (note that $N_A=N_S=0$ for $n<8$).

\begin{center}
\begin{tabular}{|c|c||c|c||c|c||c|c||c|c|}
\hline
$n$&$N_{SA}$&$n$&$N_{SA}$&$n$&$N_{SA}$&$n$&$N_{SA}$&$n$&$N_{SA}$\\
\hline
8&1&18&2&28&6&38&10&48&15\\
10&1&20&3&30&6&40&11&50&17\\
12&1&22&4&32&7&42&12&&\\
14&2&24&4&34&8&44&13&&\\
16&2&26&5&36&9&46&14&&\\
\hline
\end{tabular}
\end{center}

\item[Zero-free fixed points.] There is one triad fixed point for any digit-count $n$ that is divisible by $3$.

We can combine the enumeration of uniform zero-free fixed points with class (a) of other zero-free fixed points: together, these fixed points exist for any pair of integers $(t,u)$ with $1\le u\le t$ and $5t+2u=n$. Equivalently, each of these fixed points corresponds to a pair of integers $(t,v)$ with $t\ge0$, $v\ge1$, and $5t+7v=n$. Hence the number of zero-free fixed points when $n\equiv r$ (mod $35$) is
$$N_Z=\left\lfloor\frac{n}{35}\right\rfloor+q,$$
where $q=0$ for $r=0,1,2,3,4,5,6,8,9,10,11,13,15,16,18,20,23,25,30$, and $q=1$ for all other positive $r\le34$.

The other zero-free fixed points of class (b) occur when $n=7t$ for any $t>1$, with the number of such fixed points being $n/7-1$.
\item[Other fixed points.] There is one single-parameter fixed point for any digit-count $n$ that is divisible by $6$. The only special fixed point occurs with $n=2$.
\item[Symmetric and almost-symmetric cycles.]

In base $8$ we have $B=7$ and $C=3$ in the notation used in Subsection \ref{enumsasc}; as in base $6$, both $B$ and $C$ are prime, with $\sigma(B)=C$, so formulae in Subsection \ref{enumsasc} can be used without any additional computations. From (\ref{ngamt}) and (\ref{ngamma}) ,
\beq\label{ngam8}
N_\gamma(3,\omega)=\left\lfloor\frac{(\omega-1)(\omega-2)}{6}\right\rfloor.
\eeq
Evaluating the summation in (\ref{neta}) for the cases $\eta\equiv\{0,1,2\}$ (mod $3$), we find that the formula
\beq\label{seta8}
N_{k\gamma}(3,\eta)=\left\lfloor\frac{\eta(\eta-3)^2}{18}\right\rfloor
\eeq
applies in all cases. The number of symmetric cycles with digit-count $n$ is then
$$N_{SC}=N_{k\gamma}\left(3,\frac{n}{2}\right)=\left\lfloor\frac{n(n-6)^2}{144}\right\rfloor$$
for even $n$ (and $N_{SC}=0$ for odd $n$). The numbers of almost-symmetric cycles with odd and even $n$ are obtained respectively from (\ref{nacodd}) and (\ref{naceven}) as
$$N_{AC}=\sum_{m=1}^{\lfloor(n-3)/14\rfloor}\left\lfloor\frac{(n-14m+7)(n-14m+1)^2}{144}\right\rfloor$$
and
$$N_{AC}=\sum_{m=1}^{\lfloor(n-10)/14\rfloor}\left\lfloor\frac{(n-14m)(n-14m-6)^2}{144}\right\rfloor.$$
As with the symmetric and almost-symmetric fixed points, closed-form expressions for the above sums depend on congruence of $n$ modulo $42$. Noting equation (\ref{ascevenodd}), which in base $8$ takes the form
\beq\label{nansac8}
N_{AC}(n+14)=N_{AC}(n+7)=N_{SAC}(n)\quad\text{for even $n$},
\eeq
we have evaluated $N_{SAC}=N_{SC}+N_{AC}$ and we find that the following formula applies in all cases:
$$N_{SAC}=\left\lfloor\frac{n(n+6)(n^2+6n-104)}{8064}+\frac{3}{7}\right\rfloor\quad\text{for even $n$}.$$
Results of a numerical evaluation of this formula for $10\le n\le50$ are tabulated below.

\begin{center}
\begin{tabular}{|c|c||c|c||c|c||c|c||c|c|}
\hline
$n$&$N_{SAC}$&$n$&$N_{SAC}$&$n$&$N_{SAC}$&$n$&$N_{SAC}$&$n$&$N_{SAC}$\\
\hline
10&1&20&27&30&131&40&396&50&936\\
12&3&22&39&32&168&42&478&&\\
14&6&24&55&34&212&44&572&&\\
16&11&26&75&36&264&46&679&&\\
18&18&28&100&38&325&48&800&&\\
\hline
\end{tabular}
\end{center}

The value of $N_A$ then follows from (\ref{nansac8}) in every case where $N_A$ is non-zero.
\item[Non-symmetric $\sigma$-cycles.] From the specification of these cycles in (\ref{nssigma8}), $k_1\ge0$, $k_2\ge1$, and \mbox{$k_3\ge1$}; so, the number of triples of integers satisfying these requirements and adding to a given $\omega$ is equal to $\widetilde{N_\gamma}(3;\omega+1)$, where the specification of $\widetilde{N_\gamma}$ requires the integers in the triple to each be strictly positive. Now, $n=2(k_1+k_2+k_3)+2$ from (\ref{nssigma8}), so $\omega:=k_1+k_2+k_3=(n-2)/2$. Hence the number of non-symmetric $\sigma$-cycles is given by (\ref{ngam8}) as
$$N_{N\sigma}=N_\gamma\left(3,\frac{n}{2}\right)=\left\lfloor\frac{(n-2)(n-4)}{24}\right\rfloor$$
for even $n$, and zero for odd $n$.
\item[Single-parameter cycles.] From the list of these cycles in Subsection \ref{lfpc8}, they occur as follows.
\begin{itemize}
\item One cycle when $n\equiv0$ (mod $2$) with $n\ge4$;
\item  One cycle when $n\equiv2$ (mod $5$) with $n\ge12$;
\item  One cycle when $n\equiv1$ (mod $7$) with $n\ge15$;
\item  One cycle when $n\equiv3$ (mod $7$) with $n\ge17$;
\item  Two cycles when $n\equiv4$ (mod $7$) with $n\ge11$;
\item  Two cycles when $n\equiv6$ (mod $7$) with $n\ge13$.
\end{itemize}
\item[Special cycles] These occur as follows. One at $n=4$; two at $n=5$; two at $n=7$; three at $n=9$; one at $n=11$; one at $n=12$.

\end{description}

\section{Conclusions}\label{sec6}

Based on data provided by Joseph Myers in the OEIS for bases $4$, $6$ and $8$, we have developed a classification of the fixed points and cycles of the Kaprekar transformation in even bases. The symmetric and almost-symmetric fixed points and cycles, and the uniform zero-free and triad fixed points, must exist in every even base $b\ge4$. However, there also exist other classes of fixed points and cycles, which become increasingly prominent in higher bases; so we cannot be certain that our classification is complete. We have provided general formulae and methods for enumerating [almost-]symmetric fixed points and cycles, but the complications of multiple ``$i$-cycles'' mean that further work will be needed for such enumerations in most even bases $b>8$. The symmetric and almost-symmetric fixed points and cycles are far more numerous than any other class. In contrast to our findings in odd bases, special cycles which do not fit within our classification are rather rare in even bases, and exist only for low digit-counts. Nevertheless, they do become more numerous in higher bases.

Since our interest is in cycles, it is not surprising that cyclic groups should play an important role in the theory. However, it does seem remarkable that the subgroup generated by $2$ in the multiplicative group modulo some $r$ plays such a pivotal role in both the regular cycles in odd bases \cite{kdw1} and the [almost-]symmetric cycles in even bases, yet in completely different ways. In the former it determines the length of cycles according to their digit-count, whereas in the latter it determines the length according to the base. Hence it appears that cycles of arbitrary length may exist in any odd base, whereas there may be a maximum cycle length in any given even base; but these conjectures remain unproved.

\section{Acknowledgments}

This work was begun by KD-W as a final-year undergraduate project under the supervision of AK at Loughborough University, and continued by AK after his retirement. The authors thank Ben Crossley for useful discussions. The work could not have been done without the computational results made freely available by Joseph Myers in the OEIS.

\bigskip
\hrule
\bigskip

\noindent 2020 {\it Mathematics Subject Classification}:
11A99.

\noindent \emph{Keywords: } 
Kaprekar transformation, cycle, fixed point.

\end{document}